\documentclass[11pt]{article}
\usepackage{amsmath,amsthm}
\usepackage{amsfonts}
\usepackage{mathrsfs}
\usepackage{hyperref}
\usepackage{cite}
\usepackage{booktabs}
\usepackage{cleveref}
\usepackage{amssymb}
\usepackage{caption}
\usepackage{subcaption}
\usepackage{graphicx}
\usepackage{xcolor}
\usepackage{bm}
\usepackage{lineno}
\usepackage{geometry}

\usepackage{listings}%
\usepackage{multicol,multirow}
\usepackage{tabularx,booktabs}
\usepackage{boites,enumerate}

\crefformat{figure}{#1~(#3)}
\numberwithin{equation}{section} 
\newtheorem{theorem}{\bf Theorem}[section]
\newtheorem{example}{\bf Example}[section]
\newtheorem{corollary}{\bf Corollary}[section]
\newtheorem{remark}{\bf Remark}[section]
\newtheorem{lemma}{\bf Lemma}[section]

\def\theconclusion{\theconclusion.\arabic{conclusion}}

\newcommand{\norm}[1]{\left\lVert #1\right\rVert}

\newsavebox{\savepar}

\begin{document}
	
	\title{\vspace*{-46pt}\bf Finite Difference Method for Global Stabilization of the Viscous Burgers' Equation with Nonlinear Neumann Boundary Feedback Control
	}
	
\author{
	Shishu Pal Singh\thanks{Department of Mathematical Sciences, Rajiv Gandhi Institute of Petroleum Technology. Email: shishups@rgipt.ac.in}
	\;and
Sudeep Kundu\thanks{Department of Mathematical Sciences, Rajiv Gandhi Institute of Petroleum Technology. Email: sudeep.kundu@rgipt.ac.in} 
}

	\maketitle
	
	\begin{abstract}
		This article focuses on a nonlinear Neumann boundary feedback control formulation for the viscous Burgers' equation and develops a class of finite difference schemes to achieve global stabilization. The proposed procedure, known as the $\theta$-scheme with $\theta \in [0,1]$, unifies explicit and implicit time discretizations and is suitable for handling the nonlinear boundary feedback control problem. Using the discrete energy method, we prove that the proposed difference scheme is conditionally stable for $0 \leq \theta < \frac{1}{2}$ and unconditionally stable for $\theta \geq \frac{1}{2}$. In addition, we establish the exponential stability of the fully discrete solution. The error analysis shows a first-order convergence rate of the state variable in the discrete $L^{2}$-, $H^{1}$-, and $L^{\infty}$-norms for $\theta \geq \frac{1}{2}$, while preserving the exponential stability property. A first-order convergence rate for the boundary control inputs is also obtained. Numerical experiments are conducted to validate the theoretical findings and to demonstrate the effectiveness of the method for the inhomogeneous nonlinear Neumann boundary feedback control of the viscous Burgers' equation.
		
		
	\end{abstract}
	
	\section{Introduction}
	\label{intro}
		We take the viscous Burgers' equation with Neumann boundary control
	
	\begin{align}
		\label{Eq1}
		y_{t}-\nu y_{xx} + yy_{x} &= 0, \;\;\;\; (x,t)\in(0,1)\times(0,\infty) , \\
		y_{x}(0,t) &= u_{0}(t),  \;\;\;\; t\in(0,\infty),  \\
		y_{x}(1,t) &= u_{1}(t), \;\; \;\; t\in(0,\infty),	 \\
		y\left(x,0\right) &= y_{0}(x), \;\; \; \;x\in(0,1),
		\label{Eq2}
	\end{align}
	where $\nu>0$ is the viscosity coefficient; $u_{0}(t)$ and  $u_{1}(t)$ are control inputs.  For simplicity, we refer to the viscous Burgers' equation as Burgers' equation \cite{MR930029}.
	Over the past four decades, the control problem for this equation, which appears in various physical situations has been extensively studied in the literature. In thermal and chemically reacting dynamics problems, Neumann actuation is the natural choice. The steady-state solution of the Burgers' equation with zero Neumann boundary conditions is constant and not asymptotically stable  \cite{byrnes1993boundary}. Therefore, the issue of feedback stabilization becomes crucial to make the equilibrium solution asymptotically stable. Research on local stabilization has been conducted (see, e.g., \cite{burns1991control,byrnes1998global, byrnes1993boundary}, and \cite{ly1997distributed}). Regarding the optimal control problem of the Burgers' equation, feedback control has been derived by minimizing the cost functional through linearization \cite{burns1991control, ly1997distributed}. Local stabilization of this equation is achieved using both distributed and boundary control with a sufficiently small initial condition in the $L^{2}-$norm. Additionally, the optimal control problem has been analyzed using proper orthogonal decomposition \cite{kunisch1999control}. For the finite element method applied to the optimal control problem of the steady-state Burgers' equation, refer to \cite{merino2016finite}. Regarding distributed optimal control analysis and Neumann boundary, see \cite{volkwein2001distributed, sabeh2016distributed, MR2301948}. 
	
	Global stabilization for the Burgers' equation has been addressed using nonlinear  boundary feedback control \cite{krstic1999global}. In \cite{balogh2000burgers}, the authors demonstrate both global asymptotic stability and semi-global exponential stability in the Sobolev space $H^{1}$. In \cite{liu2001adaptive}, $H^{1}$ global stability and well-posedness with Neumann boundary control are established via the Lyapunov approach when the parameter $\nu$ is unknown, i.e., for adaptive control. Furthermore, \cite{smaoui2005boundary} analyzes the stability of the forced Burgers' equation with distributed control, considering both periodic and Neumann boundary conditions. In \cite{ito1998viscous}, global existence and uniqueness are established for the forced Burgers' equation when the initial data and forcing function belong to $L^{\infty}$, under nonlinear Neumann boundary feedback control.  
	Moreover, \cite{gumus2022finite} examines global exponential stabilization using distributed control. 
	
	Extensive research has focused on this equation with Dirichlet boundary conditions, significantly advancing our understanding of its behavior and analytical characteristics. In \cite{kadalbajoo2006numerical}, the equation is solved using a Crank-Nicolson scheme with the aid of the Cole-Hopf transformation, which converts it into a heat equation with Neumann boundary conditions. A similar approach using a finite element method is detailed in \cite{ozics2003finite}. Under the Cole-Hopf transformation, the Burgers’ equation with a nonlinear Neumann boundary condition is transformed into a heat equation with a nonlinear Robin boundary condition\cite{MR3100771, nguyen2001numerical}.
	Without the Cole-Hopf transformation, a Crank-Nicolson method is used to solve the equation in \cite{wani2013crank}. Additionally, \cite{jiwari2013numerical} examines the equation through the weighted average differential quadrature method in space and a finite difference method in time. In \cite{MR4190973}, it is demonstrated that when the initial condition lies in $H^{1}$, an {\it a priori} estimate of the analytic solution holds in the $L^{\infty}$ sense. Furthermore, the authors provide the stability and convergence analysis of a finite difference scheme using the discrete energy method.
	
	In \cite{MR4233226}, the stability and convergence analysis is discussed for the Burgers' equation using a compact finite difference scheme. This scheme, known for its high-order accuracy and implicit nature, efficiently approximates derivatives using compact stencils, minimizing the number of neighboring grid points while preserving high accuracy. Utilizing this method, \cite{sun2009compact} examines the heat equation with nonzero Neumann boundary conditions, demonstrating stability and convergence through the discrete energy method.  
	
The stability analysis of difference schemes with homogeneous mixed boundary conditions was established in \cite{MR0699843}. Later, \cite{MR4242164} investigated the stability and convergence of a fourth-order compact finite difference scheme for the Benjamin–Bona–Mahony–Burgers equation. Additionally, using the cut-off function method, \cite{MR4149614} analyzed the convergence of two difference schemes—namely, a two-time-level implicit scheme and a three-time-level linearized scheme—for the viscous generalized Burgers' equation with a nonlinear term of the form  \(y^{p}y_{x} (p\geq 1 \)).
	
	Numerous studies have focused on numerical simulations without conducting stability and convergence analysis of control for this equation. In \cite{balogh2000burgers}, it is demonstrated that the uncontrolled system (with zero Neumann boundary) converges to a constant steady-state solution of the problem (\ref{Eq1}) - (\ref{Eq2}) through the use of nonlinear boundary control. The authors have employed a third-order quadratic approximation in time and a second-order finite difference approach in space. The asymptotic behavior with feedback control is investigated in \cite{ito1994dissipative}, where the Fourier-collocation method is used for spatial discretization and the two-step implicit method for temporal discretization under periodic boundary conditions. For Dirichlet boundary conditions, a Galerkin method with linear basis functions is employed. Furthermore, \cite{dean1991pointwise} demonstrates the Burgers' equation with pointwise control associated with Dirac measures, using a finite element method in space and a finite difference method in time.  
	
	A finite difference method for handling the nonlinear Neumann boundary condition is investigated in \cite{pao1995finite} for a nonlinear reaction-diffusion equation using the monotone iterative technique. In \cite{ferreira2024fdm}, the parabolic reaction-diffusion equation with a zero Neumann boundary is discussed, with stability and convergence analysis established through the discrete energy method. Additionally, \cite{kundu2018finite} applies a conforming finite element analysis to the Burgers' equation with nonlinear Neumann boundary feedback control, examining convergence results and optimal error estimates in $L^{\infty}(L^{2})$, $L^{\infty}(H^{1})$, and $L^{\infty}(L^{\infty})$. The stability and instability of this equation with nonlinear Neumann boundaries were analyzed in \cite{MR930029}. Various numerical methods have been proposed for solving it under Neumann or mixed boundaries, including the finite element method \cite{nguyen2001numerical, smith1997finite, pugh1995finite}, Adomian decomposition \cite{BAKODAH20171339}, probabilistic approaches \cite{milstein2002probabilistic}, and localized collocation methods \cite{NOJAVAN201823}. Additionally, a three-time-level splitting technique combined with cubic splines was applied to the nonlinear Neumann boundary case in \cite{MR1160632}.
	
	In this study, we delve into a theta scheme tailored for addressing the Burgers' equation with nonlinear feedback control at the boundary. The principal objective is to establish the exponential stability and conduct error analysis for a nonlinear difference scheme accommodating a nonlinear nonhomogeneous Neumann boundary condition, utilizing the discrete energy method. The inclusion of a nonlinear term at the boundary adds complexity to the stability and error analysis of the difference scheme. We demonstrate that despite employing a second-order accurate interior difference, the scheme attains first-order accuracy due to the use of a first-order Neumann boundary condition. Moreover, the theta scheme has been applied in numerical investigations of various equations, including the complex diffusion equation \cite{araujo2012stability, araujo2015stability} and the generalized Fisher's equation \cite{esmailzadeh2022numerical}. The author in \cite{MR2099313} has discussed stability and error analysis for a theta scheme to the lambda-omega system with zero Neumann and Dirichlet boundaries. To the author's knowledge, this study marks the first instance of designing a theta scheme specifically tailored for the Burgers' equation with nonlinear feedback Neumann boundary control laws.
	
	The main contributions of this paper are as follows:
	
	\begin{itemize}
		\item We discuss a theta scheme for this equation with a suitable discretization for the nonlinear term, including the $w_x$ term.
		\item The exponential stability of the difference scheme is established using the discrete energy method. Specifically, the system is unconditionally stable for $\theta \in [\frac{1}{2}, 1]$, while for $\theta \in [0, \frac{1}{2})$, the system is conditionally stable. Moreover, {\it{a priori}} exponential bounds of the state variable are shown in \(L^{2}, H^{1}\) and \(L^{\infty}\)-norms for $\theta \in [\frac{1}{2}, 1]$.
		\item The error analysis to both the state variable and control inputs is discussed for $\theta \in [\frac{1}{2}, 1]$. Additionally, the error estimate of the state variable is discussed in \(L^{2}, H^{1}\) and \(L^{\infty}\)-norms for $\theta \in [\frac{1}{2}, 1]$.
		\item We conduct some numerical experiments to demonstrate the behavior of the control and state trajectories, including their order of convergence in various norms.
	\end{itemize}

	It is worth noting, as demonstrated in \cite{kundu2018finite}, that a steady-state solution of the problem (\ref{Eq1}) - (\ref{Eq2}) with a zero Neumann boundary condition is any constant, denoted as $w_{d}$. Note that without control inputs i.e. $u_0(t)=0$, $u_1(t)=0$, the approximate solution of \eqref{Eq1}-\eqref{Eq2} converges to a non constant steady state solution as $ t\rightarrow \infty,$ when $ \nu $ is small and $ y_{0}(x) $ in \eqref{Eq2} is antisymmetric and not too small ; for details, see \cite{MR1910446, MR3100771}  and \cite{MR1247468}. In our case, when $ \nu $ is small but fixed, the approximate solution of \eqref{Eq1}-\eqref{Eq2} goes to its constant steady state solution with the help of boundary control. Without loss of generality, we assume $w_{d}\geq 0$. To ensure that $\lim_{t\to \infty}y(t,x)=w_{d}$ for all $x\in[0,1]$, we introduce the transformation $w=y-w_{d}$, such that $w=0$ as $t \to \infty$.\\
	By substituting (\ref{Eq1}) - (\ref{Eq2}), we establish the following system:
	\begin{align}
		\label{w11}
		w_{t}-\nu w_{xx} + w_{d}w_{x}+ ww_{x} &= 0, \;\;\;\; (x,t)\in(0,1)\times(0,\infty) , \\
		w_{x}(0,t) &= v_{0}(t),   \;\;\;\; t\in(0,\infty),  \\
		w_{x}(1,t) &= v_{1}(t),  \;\; \;\; t\in(0,\infty),	 \\
		w\left(x,0\right) &= y_{0}(x)-w_{d} = w_{0}(x) \ (say),\;\; \; \;x\in(0,1),
		\label{w12}
	\end{align}
	where $v_{0}(t)$ and $v_{1}(t)$ represent the feedback control laws, as specified in \cite{kundu2018finite}, given by:
	\begin{align*}
		v_{0}(t)&= \frac{1}{\nu}\left( (c_{0}+w_{d}) w(0,t) + \frac{2}{9c_{0}}w(0,t)^{3}\right),\\
		v_{1}(t)&= -\frac{1}{\nu}\left( (c_{1}+w_{d})w(1,t)+\frac{2}{9c_{1}}w(1,t)^{3}\right),
	\end{align*}
	where $c_{0}$ and $c_{1}$ are positive constants. These controllers can be used to design the optimal control \cite{krstic1999global}. Such controller are also known as inverse optimal feedback controllers.
	
	In this paper, we adopt the following notations:
	\begin{itemize}
		\item \cite{evans1998partial} The space $ L^{p}((0,T);X)$ consists of all measurable function $ \phi:[0,T]\rightarrow X $, equipped with the norm:
		\begin{align*}
			\norm{\phi}_{L^{p}((0,T);X)}=
			\begin{cases}
				\left(\int_{0}^{T}\norm{\phi}_{X}^{p}dt\right)^{\frac{1}{p}}< \infty, &\text{if} \ 1 \leq p<\infty, \\
				\mathop{\mathrm{ess\, sup}} \limits_{0\leq t\leq T} (\norm{\phi(t)}_{X})< \infty, \medspace &\text{if}\  p=\infty,
			\end{cases}
		\end{align*}
		where $ X $ denote a Banach space with norm $ {||.||}_{X}$.
		\item \cite{evans1998partial}The Sobolev space $H^{m}(0,1)=\{\phi \vert \ \phi\in L^{2}(0,1),\ \frac{\partial^{\alpha} \phi}{\partial x^{\alpha}}\in L^{2}(0,1), \text{where} \ 1\leq \alpha\leq m $ \}.
		\item  $ C^{n,m}([0, 1], [0, T])$ denotes the space of functions defined in $ [0, 1]\times[0, T]$,
		with continuous partial derivatives  of orders at most  $ n $ and $ m $, with respect to $ x $ and $ t $ respectively.
		\item \textbf{Young's inequality:} For each $ a,b >0$ and $\epsilon>0$
		\begin{align*}
			ab\leq \epsilon\frac{a^{p}}{p}+\frac{1}{\epsilon^{\frac{q}{p}}}\frac{b^{q}}{q},
		\end{align*}
		where $ 1<p,q<\infty $ and $ \frac{1}{p} +\frac{1}{q} =1.$
		\item \cite{heywood1990finite} \textbf{Discrete Gronwall's inequality:} Let $k, B$ and $a_{n}, b_{n}, c_{n} $ for integers $ n\geq 0 $ be nonnegative numbers such that
		\begin{align*}
			a_{n}+k\sum_{l=0}^{n}b_{l}\leq k\sum_{l=0}^{n}c_{l}+C k \sum_{l=0}^{n}a_{l}+B, \hspace{0.5cm} \ \text{for} \ n\geq0,
		\end{align*}
		where C is a positive constant.
		Suppose that $ Ck<1 $. Then,
		\begin{align*}
			a_{n}+k\sum_{l=0}^{n}b_{l}\leq e^{\frac{Ck(n+1)}{(1-Ck)}}\left(k\sum_{l=0}^{n}c_{l}+B\right), \ \ \text{where} \  C>0.
		\end{align*}
	\end{itemize}

	The structure of the rest of the paper is as follows: In Section $2$, we construct a theta scheme for the problem (\ref{w11}) - (\ref{w12}) and introduce preliminary notations, ensuring the consistency of a difference scheme. In Section $ 3 $, the existence and exponential stability of the difference scheme are analyzed. The stability condition is proven for $\theta \in [0, 1]$.  In Section $ 4 $, the error analysis of the difference scheme is discussed, when $\theta \in [\frac{1}{2}, 1]$. Finally, in Section $5$, we present some numerical examples that validate our theoretical results.
	\section{Fully Discrete Scheme.}
	Let $ M $ and $ N $ be two positive integers, and let $ h=\frac{1}{N}$ be the spatial step size. We define the uniform mesh points on the unit interval $ [0, 1] $ as follows: $$ x_{i}= ih, \ \ \text{for} \  i=0, 1, 2,\ldots, N.$$ Similarly, we define the  uniform mesh points on the interval $ [0, T] $  by $$ t_{n}=nk,\ \  \text{for}\  n=0, 1,\ldots, M, $$ where $ k=\frac{T}{M} $ is the temporal step size. Denote $ t_{n+\theta}:=\theta t_{n+1}+(1-\theta)t_{n}$, where  $ \theta\in[0, 1] $.
	
	Let $W_{i}^{n}\approx w(x_{i},t_{n})$ represents an approximate solution at the point $\left(x_{i}, t_{n}\right)$ for the problem (\ref{w11}) - (\ref{w12}). We define the difference operators as follows: 
	\begin{align*}
		\delta_{t}^{+} W_{i}^{n} &:=\dfrac{W_{i}^{n+1} - W_{i}^{n}}{k} ,\;\;\; 	\delta_{x}^{c}W_{i}^{n} :=\dfrac{W_{i+1}^{n} - W_{i-1}^{n}}{2h},	\\	\delta_{x}^{-}W_{i}^{n} &:=\dfrac{W_{i}^{n} - W_{i-1}^{n}}{h},  \;\;\; \delta_{x}^{2} W_{i}^{n} :=\dfrac{W_{i+1}^{n} -2 W_{i}^{n}+W_{i-1}^{n} }{h^{2}}, \ \ \delta_{x}^{+}W_{i}^{n} :=\dfrac{W_{i+1}^{n} - W_{i}^{n}}{h},\\
		W_{i}^{n+\theta}& := \theta W_{i}^{n+1} + (1-\theta) W_{i}^{n} , \hspace{3mm} \text{where} \hspace{3mm}\theta \in [0,1], \medspace \norm{W^{n}}_{\infty}:= \max_{0\leq i\leq N}|W_{i}^{n}|.
	\end{align*}
	Let $ H=\left\{W| W=(W_{0},W_{1}\cdots W_{N})\right\}$ denote the space of grid functions.
	
	We define a grid function $ \phi: H\times H \rightarrow H $ by
	\begin{align}\label{eq2.0}
		\phi(W,W)_{i}:=\begin{cases} \frac{1}{3}(2W_{0}+W_{1})\delta_{x}^{+}W_{0}, &  i=0, \vspace{0.3cm}\\  \frac{1}{3}(W_{i-1}+W_{i}+W_{i+1})\delta_{x}^{c}W_{i}, &   i=1, 2,\ldots, N-1,\vspace{0.3cm} \\ \frac{1}{3}(2W_{N}+W_{N-1}) \delta_{x}^{-}W_{N}, &  i=N.  \end{cases}
	\end{align}
	The distinctive discretization method for the nonlinear term $ww_{x}$ has been employed in prior works such as those cited in \cite{fornberg1973instability, chung1998finite}, and \cite{omrani2008finite}. The discretization approach for the nonlinear term $\phi(w, w)$ is derived by expressing it as $\frac{1}{3}(ww_{x} + (w^{2})_{x})$. However to handle the nonlinear term in stability and error analysis we introduce a different discretization for the boundary case mentioned in \eqref{eq2.0}. 
	
	For any grid function $ W \in H $, we define the operator $\delta_{x}^{2} W_{i}$ as follows
	\begin{align*}
		\delta_{x}^{2}W_{i}:=
		\begin{cases}
			\frac{2}{h}(\delta_{x}^{+}W_{0}-g_{0}(W_{0})), & \ i=0, \vspace{0.3cm}\\
			\frac{1}{h^{2}}(W_{i+1}-2W_{i}+W_{i-1}), &\ \ i=1,2,\ldots, N-1,\vspace{0.3cm}\\
			\frac{2}{h}(-\delta_{x}^{-}W_{N}+g_{N}(W_{N})), &\  i=N. 
		\end{cases}
	\end{align*}
	For this type of discretization of $ \delta_{x}^{2}W ,$ see \cite[Chapter 2]{samarskii2001theory},
	where 
	\begin{align}
		g_{0}(W_{0}^{n}) &=\frac{1}{\nu}\left((c_{0}+w_{d})W_{0}^{n} + \frac{2}{9c_{0}}(W_{0}^{n})^{3}\right),\medspace \  n=0,1,2,\ldots,M,\label{eq1}\\		
		g_{N}(W_{N}^{n}) &= -\frac{1}{\nu}\left((c_{1}+w_{d})W_{N}^{n} + \frac{2}{9c_{1}}(W_{N}^{n})^{3}\right),\medspace \  n=0,1,2,\ldots,M.\label{eq2}
	\end{align}
	We approximate the problem (\ref{w11}) - (\ref{w12}) using a theta scheme to find $ W_{i}^{n} $ such that for all $ 0\leq n \leq  M-1  $
	\begin{align}
		\label{th}
		\begin{cases}
			\delta_{t}^{+}W_{0}^{n}-\frac{2\nu}{h^{2}}\left[W_{1}^{n+\theta}-W_{0}^{n+\theta}-hg_{0}(W_{0}^{n+\theta})\right]+w_{d}\delta_{x}^{+}W_{0}^{n+\theta}+\phi(W,W)_{0}^{n+\theta}=0,\ i=0 ,\  \vspace{0.2cm} \\
			\vspace{0.2cm}
			\delta_{t}^{+} W_{i}^{n}-\nu \delta_{x}^{2} W_{i}^{n+\theta}+ w_{d}\delta_{x}^{c} W_{i}^{n+\theta}+ \phi(W^{n+\theta},W^{n+\theta})_{i} =0, \ \  i=1,2,\ldots, N-1,  \\
			\delta_{t}^{+}W_{N}^{n}-\frac{2\nu}{h^{2}}[W_{N-1}^{n+\theta}-W_{N}^{n+\theta}+hg_{N}(W_{N}^{n+\theta})]+w_{d}\delta_{x}^{-}W_{N}^{n+\theta}+\phi(W,W)_{N}^{n+\theta}=0, \  i=N,\ \vspace{0.2cm} \\
			W_{i}^{0}=w_{0}(x_{i}), \medspace \ i=0, 1,\ldots, N,
		\end{cases}		
	\end{align} 
	where $	g_{0}(W_{0}^{n})$ and $g_{N}(W_{N}^{n})$ are defined in \eqref{eq1}-\eqref{eq2}.

	The scheme (\ref{th}) is referred to as explicit for $\theta = 0$, implicit for $\theta > 0$, and Crank-Nicolson type for $\theta = \frac{1}{2}$. We use  $ C $ to represent a generic positive constant, independent of the step sizes. This constant may depend on the parameters $ c_{0}, c_{1}, \nu,$ $\norm{W^{n}}_{\infty},$ and $ w_{d}.$ 
	\subsection{Consistency.}
	In this subsection, we develop the consistency of the difference scheme, which indicates that the scheme should converge to the differential equation as the step sizes in time and space approach zero.
	
	Assume that $ w(x, t)\in C^{4, 3}([0, 1], [0, T]) $ satisfying (\ref{w11}) - (\ref{w12}). To assess consistency, we follow a method akin to that presented in \cite{jovanovic2013analysis} and \cite{morton1998numerical}. By employing a Taylor series expansion centered at the point $ (x_{i}, t_{n+\frac{1}{2}}) $, we arrive at the truncation error for interior points as follows
	\begin{align}
		\label{eq3.1}
		T_{i}^{n+\theta}=
		\begin{cases}
			O(h^{2}+ k^{2}) & \text{for} \medspace \theta = \frac{1}{2}, \\  O(h^{2}+ k) & \text{for} \medspace \theta \neq \frac{1}{2},
		\end{cases}
	\end{align}
	where $ i=1, 2,\ldots, N-1 $. Further it holds that $$ 	T_{i}^{n+\theta}\rightarrow 0  \medspace \text{as}   \medspace h\rightarrow 0, \  k \rightarrow 0 . $$ 
	Hence, the finite difference scheme exhibits consistency at interior points.
	
	The truncation error at the boundary conditions  $x=0,$ and $x= 1$ is obtained as
	\begin{align}
		\label{eq3.2}	
		T_{i}^{n+\theta}= \begin{cases}
			O(h+ k^{2}) & \text{for} \medspace \theta = \frac{1}{2}, \\  O(h+ k) & \text{for} \medspace \theta \neq \frac{1}{2},
		\end{cases}
	\end{align}
	where $ i=0, N. $
	
	Thus, the difference scheme (\ref{th}) demonstrates consistency at both interior and boundary points.
	\section{Existence and Stability.}
	We introduce the discrete $ L^{2}-$inner products on the space of grid function $ H $
	\begin{align*}
		(W,V):= \frac{h}{2}W_{0}V_{0} + h \sum_{i=1}^{N-1}W_{i}V_{i}+\frac{h}{2}W_{N}V_{N}, \medspace \  \text{for}  \medspace V,W\in H,
	\end{align*} 
	and
	\begin{align*}
		(W,V)_{h}:=h \sum_{i=1}^{N}W_{i}V_{i}, \medspace \ \text{for}  \medspace V, \ W\in \hat{H} , 
	\end{align*}
	with their corresponding norms defined as
	\begin{align*}
		\norm{W}:=\left(W,W\right)^{\frac{1}{2}},\medspace  \norm{W}_{h}:=\left(W,W\right)_{h}^{\frac{1}{2}},
	\end{align*}
	where $ \hat{H} = \left\{W \ |\  W=(W_{1}, W_{2},\ldots, W_{N}) \right\}.$
	
	In the following lemma, we prove the discrete Poincar\'e inequality.
	\begin{lemma} For $ W\in H $, the following holds 
		\begin{align*}
			\norm{W}^{2}\leq (W_{0})^{2}+(W_{N})^{2}+\norm{ \delta_{x}^{-}W}_{h}^{2}.
		\end{align*}
	\end{lemma}
\begin{proof}
	For the proof see \cite[Pages 121-122]{samarskii2001theory}.
\end{proof}

	The following two lemmas assist in proving the stability and conducting error analysis for the state variable.
	\begin{lemma}
		\label{L4.1}
		For $W \in H $, the following result holds
		\begin{align*}
			\frac{h}{2}\delta_{x}^{+}W_{0}W_{0} +h \sum_{i=1}^{N-1}\delta_{x}^{c}W_{i}W_{i}+\frac{h}{2}\delta_{x}^{-}W_{N}W_{N} &=\frac{1}{2}\left(W_{N}^{2}-W_{0}^{2}\right).
		\end{align*}
		\begin{proof}
			Since 
			\begin{align*}
				h \sum_{i=1}^{N-1}\delta_{x}^{c}W_{i}W_{i}=\frac{1}{2}\left(-W_{0}W_{1}+W_{N}W_{N-1}\right),
			\end{align*}
			by adding the terms $ \frac{h}{2}\delta_{x}^{+}W_{0}W_{0} $  and $ \frac{h}{2}\delta_{x}^{-}W_{N}W_{N} $ to above equality, we obtain
			\begin{align*}
				\frac{h}{2}\delta_{x}^{+}W_{0}W_{0} +h \sum_{i=1}^{N-1}\delta_{x}^{c}W_{i}W_{i}+\frac{h}{2}\delta_{x}^{-}W_{N}W_{N} &=\frac{1}{2}\left(W_{N}^{2}-W_{0}^{2}\right).
			\end{align*}
			The proof is completed.
		\end{proof}
	\end{lemma}
	\begin{lemma}
		\label{L4.2}
		For $W\in H, $ there holds
		\begin{align*}
			\left(\phi(W,W),W \right)=\frac{1}{3}\left(W_{N}^{3}-W_{0}^{3}\right).
		\end{align*}
		\begin{proof}
			Summing  the nonlinear term for the interior points  gives
			\begin{align}
				\label{s1}
				h\sum_{i=1}^{N-1}\phi(W,W)_{i}W_{i}=\frac{1}{6}(-W_{0}^{2}W_{1}-W_{0}W_{1}^2+W_{N}^{2}W_{N-1}+W_{N}W_{N-1}^2).
			\end{align}
			Thus, utilizing the grid function $ \phi$ to incorporate  the boundary terms in the definition of discrete $ L^{2}-$inner product, it follows  that
			\begin{align*}
				\left(\phi(W,W), W \right)&=\frac{h}{6}(2W_{0}+W_{1})\delta_{x}^{+}W_{0}W_{0}+h\sum_{i=1}^{N-1}\phi(W,W)_{i}W_{i}+\frac{h}{6}(2W_{N}+W_{N-1})\delta_{x}^{-}W_{N}W_{N},\\
				&=\frac{1}{3}\left(W_{N}^{3}-W_{0}^{3}\right).
			\end{align*}
		\end{proof}
	\end{lemma}
	
	\subsection{Existence}
	We establish the existence of a solution to the difference scheme \eqref{th} by employing Brouwer's fixed point theorem. 
	\begin{theorem}(Brouwer's fixed point  theorem \cite{browder1965existence}).
		\label{th4.1}
		Let $X$ be a finite dimensional inner product space. Suppose that $ f:X \rightarrow X $ is continuous and there exists $ \alpha >0$ such that $ (f(x),x)\geq 0 \medspace \forall x\in X $ with $ \norm{x}=\alpha. $ Then, there exists $ x_{1}\in X $ such that $ f(x_{1})=0 $ and $ \norm{x_{1}} \leq \alpha.$
	\end{theorem}
	
	\begin{theorem}
		The difference scheme (\ref{th}) has at least a solution.
	\end{theorem}
	
	\begin{proof}
		A theta scheme is a two-time-level nonlinear difference scheme that involves computations at the current time level and the next time level. Let $W^n$ represent the value at the $n^{th}$ time level, assumed to be uniquely determined. The  scheme can be seen as a system of nonlinear equations for the value $W^{n+\theta},$ where $ \theta\in[0,1]$. Once $W^{n+\theta}$ is obtained, $W^{n+1}$ is computed as:
		$$W^{n+1} =\frac{1}{\theta}(W^{n+\theta} - (1-\theta)W^n),  \text{for } \theta \neq 0. $$  For \( \theta = 0 \), the scheme is explicit and hence, at each time level \( t_{n} \), \( W^{n} \) is already known.
		
		For $ \theta\neq 0, $ we express the difference scheme (\ref{th}) as
		\begin{align}
			\label{3.2}
			\frac{1}{\theta}(W_{i}^{n+\theta}-W_{i}^{n})- k\nu \delta_{x}^{2} W_{i}^{n+\theta}+ k w_{d}\delta_{x} W_{i}^{n+\theta}+ k\phi \left(W^{n+\theta},W^{n+\theta}\right)_{i}=0,\medspace \  i=0,1,\ldots, N,
		\end{align}
		
		where
		
		for $ i=0, $  $ i=N, $ and $ i=1, 2, \ldots, N-1, $ we have  $ \delta_{x} W_{0}^{n+\theta}=\delta_{x}^{+} W_{0}^{n+\theta} $, $ \delta_{x} W_{N}^{n+\theta}=\delta_{x}^{-} W_{N}^{n+\theta},$ and $ \delta_{x} W_{i}^{n+\theta}=\delta_{x}^{c} W_{i}^{n+\theta}$ respectively in (\ref{3.2}).
		
		Let $ V=W^{n+\theta}.$ Define $ g:H\rightarrow H $ as
		\begin{align}
			\label{3.3}
			g(V)= \frac{1}{\theta}(V-W^{n}) -k\nu  \delta_{x}^{2}V + k w_{d} \delta_{x}V+ k \phi \left(V,V\right).
		\end{align}
		It is straight forward to check that $g$ is  continuous. Taking the discrete $ L^{2}-$inner product between $ (\ref{3.3}) $ and $V,$    we obtain
		\begin{align}
			\label{3.41}
			(g(V),V) =\frac{1}{\theta}(V,V) -\frac{1}{\theta} (W^{n},V) -k \nu  \left(\delta_{x}^{2}V,V\right) + k w_{d} \left(\delta_{x}V,V\right)+k\left(\phi \left(V,V\right),V\right).
		\end{align}
		Using  the Young's inequality and the Cauchy-Schwarz inequality with Lemmas \ref{L4.1} and \ref{L4.2} to the last two term on right hand side of \eqref{3.41},  we deduce
		\begin{align*}
			\nonumber (g(V),V) \geq \frac{1}{\theta}\norm{V}^{2}-\frac{1}{\theta}\norm{W^{n}}\norm{V}& +\nu k \norm{\delta_{x}^{-}V}_{h}^{2}\nonumber\\&+k\left((c_{0}+w_{d})V_{0}^{2}+\frac{2}{9c_{0}}V_{0}^{4}+(c_{1}+w_{d})V_{N}^{2}+\frac{2}{9c_{1}}V_{N}^{4}\right)\nonumber\\& +\frac{kw_{d}}{2}(V_{0}^{2}-V_{N}^{2})-\frac{kc_{1}}{2}V_{N}^{2}-\frac{k}{18c_{1}}V_{N}^{4}-\frac{kc_{0}}{2}V_{0}^{2}-\frac{k}{18c_{0}}V_{0}^{4}.
		\end{align*}
		This inequality implies
		\begin{align}
			\label{3.4}
			(g(V),V) \geq \frac{1}{\theta}\norm{V}(\norm{V}-\norm{W^{n}})&+\nu k \norm{\delta_{x}^{-}V}_{h}^{2}\nonumber\\&+\frac{k}{2}\left((c_{0}+3w_{d})V_{0}^{2}+\frac{1}{3c_{0}}V_{0}^{4}+(c_{1}+w_{d})V_{N}^{2}+\frac{1}{3c_{1}}V_{N}^{4}\right).
		\end{align}
		
		Consequently, for $ \norm{V}= 1+ \norm {W^{n}}=\alpha,$ we conclude that $(g(V),V)\geq 0,$ since the third and fourth terms on the right-hand side of \eqref{3.4} are positive. Thus, by Theorem \ref{th4.1}, there exists  $V^{*}\in H$ and $ \alpha>0 $  such that $g(V^{*})=0 $ and $ \norm{V^{*}}\leq \alpha. $
		
		This concludes the proof.
	\end{proof}
	
	\subsection{Exponential Stability Analysis.}
	Next, we deal with the exponential stability of the finite difference scheme (\ref{th}). 
	\begin{lemma}
		\label{L4.3} 
		Let $ 0\leq \alpha \leq \frac{\theta^{2}}{2}\min\{\nu, \frac{(c_{0}+w_{d})}{2},\frac{(c_{1}+3w_{d})}{2}\}$, where $\theta \in [\frac{1}{2},1]$. Assume that $ k_{0}>0 $ such that for $ 0<k\leq k_{0} $ 
		\begin{align}
			\label{3.61}
			e^{2\alpha k}\leq 1+\theta^{2} k \min\{\nu, \frac{(c_{0}+w_{d})}{2},\frac{(c_{1}+3w_{d})}{2} \}.
		\end{align}
		If $\theta \in [\frac{1}{2},1]$, then the difference scheme (\ref{th}) exhibits unconditional stability, ensuring that 	
		\begin{align*}		\norm{{W}^{M}}^{2}&+ke^{-2\alpha t_{M}}\beta^{*}\sum_{n=0}^{M-1}\Big(\norm{\delta_{x}^{-}\hat{W}^{n+1}}_{h}^{2}+(\hat{W}_{0}^{n+1})^{2}+(\hat{W}_{N}^{n+1})^{2}\Big)\\& +ke^{-2\alpha t_{M}}\theta^{4}\sum_{n=0}^{M-1} e^{2\alpha t_{n}}\left(\frac{1}{6c_{0}}(W_{0}^{n+1})^{4}+ \frac{1}{6c_{1}}(W_{N}^{n+1})^{4}\right)\leq Ce^{-2\alpha t_{M}} \norm{W^{0}}_{1}^{2},
		\end{align*}
	where $ \beta^{*} $ is given in \eqref{3.14} below and $\norm{W^{0}}_{1}^{2}:=\norm{\delta_{x}^{-}{W}^{0}}_{h}^{2}+({W}_{0}^{0})^{2}+({W}_{N}^{0})^{2}+({W}_{0}^{0})^{4}+({W}_{N}^{0})^{4}$.
	\end{lemma}
	\begin{proof}
		By multiplying the first, second, and third equations of (\ref{th}) by $ \frac{h}{2} W_{0}^{n+\theta} ,$ $ h W_{i}^{n+\theta}, $ and $ \frac{h}{2}W_{N}^{n+\theta} $ respectively, and summing from $ i=1$ to $ N-1 $ in the second equation of (\ref{th}), we employ summation by parts to obtain
		\begin{align*}
			\left( \delta_{t}^{+}W^{n},W^{n+\theta}\right) &+\nu \norm{\delta_{x}^{-}W^{n+\theta}}_{h}^{2}+\nu g_{0}(W_{0}^{n+\theta})W_{0}^{n+\theta}- \nu g_{N}(W_{N}^{n+\theta})W_{N}^{n+\theta}\\& + w_{d}\left(\frac{h}{2}\delta_{x}^{+}W_{0}^{n+\theta}W_{0}^{n+\theta} +h \sum_{i=1}^{N-1}\delta_{x}^{c}W_{i}^{n+\theta}W_{i}^{n+\theta}+ \frac{h}{2}\delta_{x}^{-}W_{N}^{n+\theta}W_{N}^{n+\theta}\right)\\&=-\frac{h}{2}\phi(W^{n+\theta},W^{n+\theta})_{0}W_{0}^{n+\theta} -h\sum_{i=1}^{N-1}\phi(W^{n+\theta},W^{n+\theta})_{i}W_{i}^{n+\theta}\\&\quad-\frac{h}{2}\phi(W^{n+\theta},W^{n+\theta})_{N}W_{N}^{n+\theta}.
		\end{align*}
		From Lemmas \ref{L4.1}  and  \ref{L4.2}, it follows from the above equality that
		\begin{align}
			\label{3.5}
			\nonumber	\left( \delta_{t}^{+}W^{n},W^{n+\theta}\right)&+\nu  \norm{\delta_{x}^{-}W^{n+\theta}}_{h}^{2}+\left((c_{0}+w_{d})(W_{0}^{n+\theta})^{2}+\frac{2}{9c_{0}}(W_{0}^{n+\theta})^{4}\right)\\& \nonumber +\left ((c_{1}+w_{d})(W_{N}^{n+\theta})^{2}+\frac{2}{9c_{1}}(W_{N}^{n+\theta})^{4}\right)+\frac{w_{d}}{2} \left( (W_{N}^{n+\theta})^{2}-(W_{0}^{n+\theta})^{2}\right)\\&=\frac{1}{3}\left( (W_{0}^{n+\theta})^{3}-(W_{N}^{n+\theta})^{3}\right).
		\end{align}
		Applying the Young's inequality on the right hand side of $ (\ref{3.5}) $ yields
		\begin{align}
			\label{3.6}
			\frac{1}{3}(W_{0}^{n+\theta})^{3}\leq \frac{c_{0}}{2}(W_{0}^{n+\theta})^{2}+\frac{1}{18c_{0}}(W_{0}^{n+\theta})^{4}\\
			\label{3.7}
			\frac{1}{3}(W_{N}^{n+\theta})^{3}\leq \frac{c_{1}}{2}(W_{N}^{n+\theta})^{2}+\frac{1}{18c_{1}}(W_{N}^{n+\theta})^{4}.
		\end{align}
		Since we can write,
		\begin{align}
			\label{3.8}
			\left( 
			\delta_{t}^{+}W^{n},W^{n+\theta}\right)=\frac{1}{2}\delta_{t}^{+}\norm{W^{n}}^{2}+k(\theta-\frac{1}{2})\norm{\delta_{t}^{+}W^{n}}^{2},
		\end{align}
		substituting $ (\ref{3.6})-(\ref{3.8})$ into (\ref{3.5}), we arrive at

		\begin{align}
			\label{3.9}
			\nonumber	\delta_{t}^{+}\norm{W^{n}}^{2}&+2k(\theta-\frac{1}{2})\norm{\delta_{t}^{+}W^{n}}^{2}+2\nu \norm{\delta_{x}^{-}W^{n+\theta}}_{h}^{2}\\&+\left((c_{0}+w_{d})(W_{0}^{n+\theta})^{2}+\frac{1}{3c_{0}}(W_{0}^{n+\theta})^{4}\right)+\left((c_{1}+3w_{d})(W_{N}^{n+\theta})^{2}+\frac{1}{3c_{1}}(W_{N}^{n+\theta})^{4}\right)\leq 0.
		\end{align}
		When $ \theta\geq \frac{1}{2} $, from $(\ref{3.9})$ we promptly deduce
		\begin{align*}
			\delta_{t}^{+}\norm{W^{n}}^{2}+2\nu\norm{\delta_{x}^{-}W^{n+\theta}}_{h}^{2}&+\left((c_{0}+w_{d})(W_{0}^{n+\theta})^{2}+\frac{1}{3c_{0}}(W_{0}^{n+\theta})^{4}\right)\\& + \left((c_{1}+3w_{d})(W_{N}^{n+\theta})^{2}+\frac{1}{3c_{1}}(W_{N}^{n+\theta})^{4}\right)\leq 0.
		\end{align*}
	Multiplying by $ e^{2\alpha t_{n+1}}$ and using 
	\begin{align}
		\label{exp}
		e^{\alpha t_{n+1}}{\delta_{t}^{+}}W^{n}=e^{\alpha k}{\delta_{t}^{+}}\hat{W}^{n}-\frac{(e^{\alpha k}-1)}{k}\hat{W}^{n+1},
	\end{align} 
	it follows that
	\begin{align}
		\label{3.121}
			e^{2\alpha k}\delta_{t}^{+}\norm{\hat{W}^{n}}^{2}&-\frac{(e^{2\alpha k}-1)}{k}\norm{\hat{W}^{n+1}}^{2}+e^{2\alpha t_{n+1}}2\nu\norm{\delta_{x}^{-}{W}^{n+\theta}}_{h}^{2}
			+e^{2\alpha t_{n+1}}\Big((c_{0}+w_{d})({W}_{0}^{n+\theta})^{2}
			\nonumber\\&+\frac{1}{3c_{0}}(W_{0}^{n+\theta})^{4}\Big)
			 + e^{2\alpha t_{n+1}}\left((c_{1}+3w_{d})(W_{N}^{n+\theta})^{2}+\frac{1}{3c_{1}}(W_{N}^{n+\theta})^{4}\right)\leq 0,
	\end{align}
where $ \hat{W}^{n}=e^{\alpha t_{n}}W^{n}.$
	
	Note that 
	\[\norm{W_{x}^{n+\theta}}_{h}^{2}=\theta^{2}\norm{W_{x}^{n+1}}_{h}^{2}+(1-\theta)^{2}\norm{W_{x}^{n}}_{h}^{2}+2\theta(1-\theta)(W_{x}^{n+1}, W_{x}^{n})_{h}, 
	\]
	and \[(W_{i}^{n+\theta})^{2}=\theta^{2}(W_{i}^{n+1})^{2}+(1-\theta)^{2}(W_{i}^{n})^{2}+2\theta(1-\theta)W_{i}^{n+1}W_{i}^{n}, \quad i=0,N.
	\]
	Also
	\begin{align*}
		(W_{i}^{n+\theta})^{4}&=\Big(\theta^{2}(W_{i}^{n+1})^{2}+ (1-\theta)^{2}(W_{i}^{n})^{2}+2\theta (1-\theta)(W_{i}^{n+1})(W_{i}^{n})\Big)^{2},
		\\&=\theta^{4}(W_{i}^{n+1})^{4}+ (1-\theta)^{4}(W_{i}^{n})^{4}+6\theta^{2} (1-\theta)^{2}(W_{i}^{n+1})^{2}(W_{i}^{n})^{2}
		\\&\quad + 4\theta^{3} (1-\theta)(W_{i}^{n+1})^{3}(W_{i}^{n})+4\theta (1-\theta)^{3}(W_{i}^{n+1})(W_{i}^{n})^{3}, \quad i=0,N.
	\end{align*}
	Therefore, from \eqref{3.121} with Young's inequality, we obtain
	\begin{align*}
	  e^{2\alpha k}\delta_{t}^{+}\norm{\hat{W}^{n}}^{2}&-\frac{(e^{2\alpha k}-1)}{k}\norm{\hat{W}^{n+1}}^{2}+\theta^{2}\nu\norm{\delta_{x}^{-}\hat{W}^{n+1}}_{h}^{2}
	+\Big(\frac{\theta^{2}(c_{0}+w_{d})}{2}(\hat{W}_{0}^{n+1})^{2}
	\nonumber\\&+\frac{\theta^{4}e^{2\alpha t_{n+1}}}{6c_{0}}(W_{0}^{n+1})^{4}\Big)
	+ \left(\frac{\theta^{2}(c_{1}+3w_{d})}{2}(\hat{W}_{N}^{n+1})^{2}+\frac{\theta^{4}e^{2\alpha t_{n+1}}}{6c_{1}}(W_{N}^{n+1})^{4}\right)
	\nonumber\\& \leq (1-\theta)^{2}e^{2\alpha t_{n+1}}\Big(2\nu\norm{\delta_{x}^{-}{W}^{n}}_{h}^{2}+(c_{0}+w_{d}) ({W}_{0}^{n})^{2}+(c_{1}+3w_{d})(W_{N}^{n})^{2}
	\\&\quad + C (1-\theta)^{2}\big((W_{0}^{n})^{4}+(W_{N}^{n})^{4}\big)\Big).
	\end{align*}
Using the discrete Poincar\'e inequality and multiplying by $ e^{-2\alpha k} $ in the resulting inequality yields
\begin{align*}
	\delta_{t}^{+}\norm{\hat{W}^{n}}^{2}&+\Big(\theta^{2}\nu e^{-2\alpha k}-\frac{(1-e^{-2\alpha k})}{k} \Big)\norm{\delta_{x}^{-}\hat{W}^{n+1}}_{h}^{2}
	+\Big(\Big(\frac{e^{-2\alpha k} \theta^{2}(c_{0}+w_{d})}{2}-\frac{(1-e^{-2\alpha k})}{k}\Big)(\hat{W}_{0}^{n+1})^{2}
	\nonumber\\&+\frac{\theta^{4}e^{2\alpha t_{n}}}{6c_{0}}(W_{0}^{n+1})^{4}\Big)
	+ \left(\Big(\frac{e^{-2\alpha k}\theta^{2}(c_{1}+3w_{d})}{2}-\frac{(1-e^{-2\alpha k})}{k}\Big)(\hat{W}_{N}^{n+1})^{2}+\frac{\theta^{4}e^{2\alpha t_{n}}}{6c_{1}}(W_{N}^{n+1})^{4}\right)
	\nonumber\\& \leq (1-\theta)^{2}e^{2\alpha t_{n}}\Big(2\nu\norm{\delta_{x}^{-}{W}^{n}}_{h}^{2}+(c_{0}+w_{d}) ({W}_{0}^{n})^{2}+(c_{1}+3w_{d})(W_{N}^{n})^{2} 
	\\&\quad + C (1-\theta)^{2}\big((W_{0}^{n})^{4}+(W_{N}^{n})^{4}\big)\Big).
\end{align*}
		Summing from $ n=0 $ to $ n=M-1,$ to the above inequality and  multiplying by $ k ,$ we get
		\begin{align*}
			\norm{\hat{W}^{M}}^{2}+k\beta^{*}\sum_{n=0}^{M-1}\Big(\norm{\delta_{x}^{-}\hat{W}^{n+1}}_{h}^{2}&+(\hat{W}_{0}^{n+1})^{2}+(\hat{W}_{N}^{n+1})^{2}\Big)\\& +k\theta^{4}\sum_{n=0}^{M-1}e^{2\alpha t_{n}} \left(\frac{1}{6c_{0}}(W_{0}^{n+1})^{4}+ \frac{1}{6c_{1}}(W_{N}^{n+1})^{4}\right)\leq C \norm{W^{0}}_{1}^{2},
		\end{align*}
	where 
	\begin{align}
		\label{3.14}
		0<\beta^{*}=\min\{\Big(\theta^{2}\nu e^{-2\alpha k}&-\frac{(1-e^{-2\alpha k})}{k} \Big), \Big(\frac{e^{-2\alpha k} \theta^{2}(c_{0}+w_{d})}{2}-\frac{(1-e^{-2\alpha k})}{k}\Big),
		\nonumber\\&\quad  \Big(\frac{e^{-2\alpha k}\theta^{2}(c_{1}+3w_{d})}{2}-\frac{(1-e^{-2\alpha k})}{k}\Big) \}.
	\end{align}
Choose $ k_{0} $ such that \eqref{3.61} holds for $ 0<k\leq k_{0}.$ The proof is completed after multiplying by $ e^{-2\alpha t_{M}}$ in the resulting inequality.
	\end{proof}
	
	The lemma presented below aids in the proof of stability and error analysis for the case $ \theta\in [0,\frac{1}{2}). $
	\begin{lemma}
		\label{L4.4} 
		For $ W \in H,$ the following results hold:
		\begin{itemize}
			\item[(1).]
			\begin{align*}
				\frac{h}{2}(\delta_{x}^{2}W_{0}^{n})^{2}+h\sum_{i=1}^{N-1}(\delta_{x}^{2}W_{i}^{n})^{2}+\frac{h}{2}(\delta_{x}^{2}W_{N}^{n})^{2}\leq \frac{6}{h^{2}}\norm{\delta_{x}^{-}W^{n}}_{h}^{2}+\frac{4}{h}(g_{0}(W_{0}^{n}))^{2}+\frac{4}{h}(g_{N}(W_{N}^{n}))^{2},
			\end{align*}
			for all $ n=0,1, \ldots , M. $
			
			\item [(2).]\begin{align*}
				\frac{h}{2}(\delta_{x}^{+}W_{0}^{n})^{2}+ h\sum_{i=1}^{N-1}(\delta_{x}^{c}W_{i}^{n})^{2}+\frac{h}{2}(\delta_{x}^{-}W_{N}^{n})^{2}\leq \norm{\delta_{x}^{-}W^{n}}_{h}^{2},
			\end{align*}	
			for all $ n=0, 1, \ldots, M. $	 
		\end{itemize}
		
	\end{lemma}

	\begin{proof}(1). For any $ W\in H,$ 
		\begin{align*}
			\frac{h}{2}(\delta_{x}^{2}W_{0}^{n})^{2}+h\sum_{i=1}^{N-1}(\delta_{x}^{2}W_{i}^{n})^{2}&+\frac{h}{2}(\delta_{x}^{2}W_{N}^{n})^{2}\\&=\frac{h}{2}(\delta_{x}^{2}W_{0}^{n})^{2}+ \frac{1}{h^{3}}\sum_{i=1}^{N-1}(W_{i+1}^{n}-2W_{i}^{n}+W_{i-1}^{n})^{2}+\frac{h}{2}(\delta_{x}^{2}W_{N}^{n})^{2},\\&=\frac{h}{2}(\delta_{x}^{2}W_{0}^{n})^{2}+\frac{1}{h^{3}}\sum_{i=1}^{N-1} \left[(W_{i+1}^{n}-W_{i}^{n})-(W_{i}^{n}-W_{i-1}^{n}) \right]^{2}+\frac{h}{2}(\delta_{x}^{2}W_{N}^{n})^{2}.
		\end{align*}
		Use of the inequality $ (a-b)^{2}\leq 2(a^{2}+b^{2}) $ to above equation, we obtain
		\begin{align*}
			\frac{h}{2}(\delta_{x}^{2}W_{0}^{n})^{2}+h\sum_{i=1}^{N-1}(\delta_{x}^{2}W_{i}^{n})^{2}&+\frac{h}{2}(\delta_{x}^{2}W_{N}^{n})^{2}\leq \frac{6}{h^{2}}\norm{\delta_{x}^{-}W^{n}}_{h}^{2}+\frac{4}{h}(g_{0}(W_{0}^{n}))^{2}+\frac{4}{h}(g_{N}(W_{N}^{n}))^{2}.
		\end{align*}
		For the proof of second part of the lemma, we obtain any $ W\in H,$
		\begin{align*}
			\frac{h}{2}(\delta_{x}^{+}W_{0}^{n})^{2}+ h\sum_{i=1}^{N-1}(\delta_{x}^{c}W_{i}^{n})^{2}+\frac{h}{2}(\delta_{x}^{-}W_{N}^{n})^{2}&=\frac{h}{2}(\delta_{x}^{+}W_{0}^{n})^{2}+\frac{h}{2}(\delta_{x}^{-}W_{N}^{n})^{2}\\&+\frac{1}{4h}\sum_{i=1}^{N-1}\left[(W_{i+1}^{n}-W_{i}^{n})-(W_{i}^{n}-W_{i-1}^{n})\right]^{2}.
		\end{align*}
		Use of the inequality $ (a-b)^{2}\leq 2(a^{2}+b^{2}) $ to above equation, we get
		\begin{align*}
			\frac{h}{2}(\delta_{x}^{+}W_{0}^{n})^{2}+ h\sum_{i=1}^{N-1}(\delta_{x}^{c}W_{i}^{n})^{2}+\frac{h}{2}(\delta_{x}^{-}W_{N}^{n})^{2}\leq \norm{\delta_{x}^{-}W^{n}}_{h}^{2}.
		\end{align*}
		The proof is completed.
	\end{proof}
	\begin{remark}
		From  \cite[Chapter 2]{samarskii2001theory}, we write
		\[\norm{W^{n}}_{\infty}^{2}\leq 2(\norm{\delta_{x}W^{n}}_{h}^{2}+ (W_{j}^{n})^{2}), \quad j=0, N.
		\]
	\end{remark}
	
	\begin{lemma}
		\label{L4.5} 
		Assume that $ 0\leq \alpha \leq \frac{\theta^{2}}{4}\min\{\beta_{1}, \beta_{2}, \beta_{4}\}, $ where $\theta \in [0,\frac{1}{2})$. Let $ k_{0}>0$, there holds for $ 0<k\leq k_{0} $
		\begin{align}
			\label{3.15}
			e^{2\alpha k}\leq 1+ \frac{k\theta^{2}}{2}\min\{\beta_{1}, \beta_{2}, \beta_{4}\}.
		\end{align}
		If $\theta \in [0,\frac{1}{2})$, then under the condition,
		$$ k< \min \{k_{1}, k_{2}, k_{3}, k_{4}, k_{5}\}, $$ where 
		\begin{align}
			\label{k42}
			\begin{cases}	
				\vspace{0.2cm}
				k_{1}= \frac{12\nu h^{2}}{(1-2\theta)(108\nu^2+ 18h^{2}w_{d}^{2}+19h^{2}\norm{W^{n+\theta}}_{\infty}^{2})}, \\\vspace{0.2cm}
				k_{2}= \frac{h}{(1-2\theta)24(c_{0}+w_{d})}, \medspace \ \ k_{3}= \frac{9c_{0}h}{(1-2\theta)32\norm{W^{n+\theta}}_{\infty}^{2}},\\
				k_{4}=\frac{h}{(1-2\theta)24(c_{1}+w_{d})}, \medspace \ \ k_{5}= \frac{9c_{1}h}{(1-2\theta)32\norm{W^{n+\theta}}_{\infty}^{2}},
			\end{cases}
		\end{align}	
		the difference scheme (\ref{th}) is  stable for all $  0\leq n\leq M-1 $
		\begin{align*}
			\norm{{W}^{M}}^{2}&+ke^{-2\alpha t_{M}}\sum_{n=0}^{M-1}\Big(\big(\frac{\theta^{2}\beta_{1}}{2}e^{-2\alpha k}-\frac{(1-e^{-2\alpha k})}{k} \big)\norm{\delta_{x}^{-}\hat{W}^{n+1}}_{h}^{2}+\big(\frac{\theta^{2}\beta_{2}}{2}e^{-2\alpha k}-\frac{(1-e^{-2\alpha k})}{k} \big)(\hat{W}_{0}^{n+1})^{2}
			\\&+\frac{\theta^{4}\beta_{3}}{2}e^{2\alpha t_{n}} (W_{0}^{n+1})^{4}+\big(\frac{\theta^{2}\beta_{4}}{2}e^{-2\alpha k}-\frac{(1-e^{-2\alpha k})}{k} \big)(\hat{W}_{N}^{n+1})^{2}+\frac{\theta^{4}\beta_{5}}{2}e^{2\alpha t_{n}}({W}_{N}^{n+1})^{4}\Big)
			\\&\leq C e^{-2\alpha t_{M}}\norm{W^{0}}_{1}^{2}.
		\end{align*}
		where 
		\begin{align}
			\label{k}
			\begin{cases}
				\beta_{1}= 2\nu-k(1-2\theta)\left(\frac{18\nu^{2}}{h^2}+ 3w_{d}^{2}+ \frac{19}{6}\norm{W^{n+\theta}}_{\infty}^{2}\right),\\
				\beta_{2}= (c_{0}+w_{d})-k(1-2\theta)\frac{24}{h}(c_{0}+w_{d})^{2}, \medspace \ \ \beta_{3}= \frac{1}{3c_{0}}-k(1-2\theta)\frac{32}{27c_{0}^{2}h}\norm{W^{n+\theta}}_{\infty}^{2},\\
				\beta_{4}=(c_{1}+w_{d})-k(1-2\theta)\frac{24}{h}(c_{1}+w_{d})^{2}, \medspace \ \ \beta_{5}= \frac{1}{3c_{1}}-k(1-2\theta)\frac{32}{27c_{1}^{2}h}\norm{W^{n+\theta}}_{\infty}^{2}.
			\end{cases}
		\end{align}
		\begin{proof}To prove stability when $ 0\leq\theta<\frac{1}{2},$ 
			from (\ref{3.9}) we can rewrite
			\begin{align}
				\label{3.10}
				\nonumber	\delta_{t}^{+}\norm{W^{n}}^{2}&+2\nu \norm{\delta_{x}^{-}W^{n+\theta}}_{h}^{2}  \\&+\left((c_{0}+w_{d})(W_{0}^{n+\theta})^{2}+\frac{1}{3c_{0}}(W_{0}^{n+\theta})^{4}\right)+\left((c_{1}+3w_{d})(W_{N}^{n+\theta})^{2}+\frac{1}{3c_{1}}(W_{N}^{n+\theta})^{4}\right)	\nonumber\\&\leq k(1-2\theta)\norm{\delta_{t}^{+}W^{n}}^{2}.
			\end{align}
			Using the definition of discrete $ L^{2}-$norm, we express
			\begin{align}
				\label{L1}
				\norm{\delta_{t}^{+}W^{n}}^{2}=\frac{h}{2}(\delta_{t}^{+}W_{0}^{n})^{2}+ h\sum_{i=1}^{N-1}(\delta_{t}^{+}W_{i}^{n})^{2}+\frac{h}{2}(\delta_{t}^{+}W_{N}^{n})^{2}.
			\end{align}

			Therefore, from the difference scheme (\ref{th}) and (\ref{L1}) with the inequality $ (a+b+c)^{2}\leq 3 (a^{2}+b^{2}+c^{2}) $, we obtain
			\begin{align}
				\label{k22}
				\nonumber	\norm{\delta_{t}^{+}W^{n}}^{2}&\leq \frac{12\nu^{2}}{2h}\left(\delta_{x}^{+}W_{0}^{n+\theta}-g_{0}(W_{0}^{n+\theta})\right)^{2}+3h\sum_{i=1}^{N-1} \nu^{2}(\delta_{x}^{2}W_{i}^{n+\theta})^{2}+\frac{12\nu^{2}}{2h}\left(\delta_{x}^{-}W_{N}^{n+\theta}-g_{N}(W_{N}^{n+\theta})\right)^{2}	\nonumber\\&+\frac{3h}{2}w_{d}^{2}(\delta_{x}^{+}W_{0}^{n+\theta})^{2}+3hw_{d}^{2}\sum_{i=1}^{N-1}(\delta_{x}^{c}W_{i}^{n+\theta})^{2}+\frac{3h}{2}w_{d}^{2}(\delta_{x}^{-}W_{N}^{n+\theta})^{2}	\nonumber\\&+\frac{3h}{2}(\phi(W^{n+\theta}, W^{n+\theta})_{0})^{2}+3h\sum_{i=1}^{N-1}(\phi(W^{n+\theta}, W^{n+\theta})_{i})^{2}+\frac{3h}{2}(\phi(W^{n+\theta}, W^{n+\theta})_{N})^{2}.
			\end{align}
			Using Lemma \ref{L4.4}, the first term on the right hand side of (\ref{k22}) is bounded by	
			\begin{align*}
				\frac{12\nu^{2}}{2h}\left(\delta_{x}^{+}W_{0}^{n+\theta}-g_{0}(W_{0}^{n+\theta})\right)^{2}&+3h\sum_{i=1}^{N-1} \nu^{2}(\delta_{x}^{2}W_{i}^{n+\theta})^{2}+\frac{12\nu^{2}}{2h}\left(\delta_{x}^{-}W_{N}^{n+\theta}-g_{N}(W_{N}^{n+\theta})\right)^{2}\\&\leq \frac{18\nu^{2}}{h^{2}}\norm{\delta_{x}^{-}W^{n+\theta}}_{h}^{2}+\frac{24}{h}\left((c_{0}+w_{d})^{2}(W_{0}^{n+\theta})^{2}+(c_{1}+w_{d})^{2}(W_{N}^{n+\theta})^{2}\right)\\&+\frac{32}{27h}\norm{W^{n+\theta}}_{\infty}^{2}\left(\frac{1}{c_{0}^{2}}(W_{0}^{n+\theta})^{4} +\frac{1}{c_{1}^{2}}(W_{N}^{n+\theta})^{4} \right).
			\end{align*}
			Combining the second term on the right hand side of (\ref{k22}) with Lemma \ref{L4.4} gives
			\begin{align*}
				\frac{3h}{2}w_{d}^{2}(\delta_{x}^{+}W_{0}^{n+\theta})^{2}+3hw_{d}^{2}\sum_{i=1}^{N-1}(\delta_{x}^{c}W_{i}^{n+\theta})^{2}+\frac{3h}{2}w_{d}^{2}(\delta_{x}^{-}W_{N}^{n+\theta})^{2}\leq 3w_{d}^{2}\norm{\delta_{x}^{-}W^{n+\theta}}_{h}^{2}.
			\end{align*}
			The third term on the right hand side of (\ref{k22}) and using $(\delta_{x}^{-}W_{0}^{n+\theta})^{2}+(\delta_{x}^{-}W_{N}^{n+\theta})^{2} \leq \frac{1}{h} \norm{\delta_{x}^{-}W^{n+\theta}}_{h}^{2} $, it follows that
			\begin{align*}
				\frac{3h}{2}(\phi(W^{n+\theta}, W^{n+\theta})_{0})^{2}+3h\sum_{i=1}^{N-1}(\phi(W^{n+\theta}, W^{n+\theta})_{i})^{2}&+\frac{3h}{2}(\phi(W^{n+\theta}, W^{n+\theta})_{N})^{2}\\&\leq\frac{19}{6}\norm{W^{n+\theta}}_{\infty}^{2} \norm{\delta_{x}^{-}W^{n+\theta}}_{h}^{2}.
			\end{align*}
			Substituting these values in (\ref{k22}), we arrive at
			\begin{align}
				\label{3.11}
				\nonumber	\norm{\delta_{t}^{+}W^{n}}^{2}&\leq \frac{18\nu^{2}}{h^{2}}\norm{\delta_{x}^{-}W^{n+\theta}}_{h}^{2}+\frac{24}{h}\left((c_{0}+w_{d})^{2}(W_{0}^{n+\theta})^{2}+(c_{1}+w_{d})^{2}(W_{N}^{n+\theta})^{2}\right)\\&+\frac{32 }{27h}\norm{W^{n+\theta}}_{\infty}^{2}\left(\frac{1}{c_{0}^{2}}(W_{0}^{n+\theta})^{4} +\frac{1}{c_{1}^{2}}(W_{N}^{n+\theta})^{4} \right)+3w_{d}^{2}\norm{\delta_{x}^{-}W^{n+\theta}}_{h}^{2} \nonumber \\&+ \frac{19}{6}\norm{W^{n+\theta}}_{\infty}^{2} \norm{\delta_{x}^{-}W^{n+\theta}}_{h}^{2}.
			\end{align}
			
			The use of $ (\ref{3.11}) $ and  $ (\ref{3.10}) $ gives
			
			\begin{align*}
				\delta_{t}^{+}\norm{W^{n}}^{2}&+\left[2\nu-k(1-2\theta)\left(\frac{18\nu^{2}}{h^2}+ 3w_{d}^{2}+ \frac{19}{6}\norm{W^{n+\theta}}_{\infty}^{2}\right) \right]\norm{\delta_{x}^{-}W^{n+\theta}}_{h}^{2}\\&+
				\left((c_{0}+w_{d})-k(1-2\theta)\frac{24}{h}(c_{0}+w_{d})^{2} \right)(W_{0}^{n+\theta})^{2}\\&+\left(\frac{1}{3c_{0}}-k(1-2\theta)\frac{32}{27c_{0}^{2}h}\norm{W^{n+\theta}}_{\infty}^{2}\right)(W_{0}^{n+\theta})^{4}\\&+
				\left((c_{1}+w_{d})-k(1-2\theta)\frac{24}{h}(c_{1}+w_{d})^{2} \right)(W_{N}^{n+\theta})^{2}\\&+\left(\frac{1}{3c_{1}}-k(1-2\theta)\frac{32}{27c_{1}^{2}h}\norm{W^{n+\theta}}_{\infty}^{2}\right)(W_{N}^{n+\theta})^{4}\leq 0.
			\end{align*}
			Therefore, we obtain
			\begin{align}
				\label{3.12}
				\delta_{t}^{+}\norm{W^{n}}^{2}+ \left(\beta_{1}\norm{\delta_{x}^{-}W^{n+\theta}}_{h}^{2}+\beta_{2}(W_{0}^{n+\theta})^{2}+\beta_{3}(W_{0}^{n+\theta})^{4}+\beta_{4}(W_{N}^{n+\theta})^{2}+\beta_{5}(W_{N}^{n+\theta})^{4}\right)\leq 0,
			\end{align}
			whenever, $ k< \min \{k_{1}, k_{2}, k_{3}, k_{4}, k_{5} \},$ where $ k_{1}, k_{2}, k_{3}, k_{4} $ and $k_{5} $ are given in (\ref{k42}), it ensure that each $ \beta_{i} ( i= 1, 2, 3, 4, 5) $ is positive, as define in (\ref{k}).
			 
			 Again, multiplying by $ e^{2\alpha t_{n+1}}$ in \eqref{3.12}, we get from \eqref{3.12} and \eqref{exp}
			 \begin{align*}
			 	e^{2\alpha k}\delta_{t}^{+}\norm{\hat{W}^{n}}^{2}-\frac{(e^{2\alpha k}-1)}{k}\norm{\hat{W}^{n+1}}^{2}&+ e^{2\alpha t_{n+1}} \Big(\beta_{1}\norm{\delta_{x}^{-}W^{n+\theta}}_{h}^{2}+\beta_{2}(W_{0}^{n+\theta})^{2}
			 	\\&+\beta_{3}(W_{0}^{n+\theta})^{4}+\beta_{4}(W_{N}^{n+\theta})^{2}+\beta_{5}(W_{N}^{n+\theta})^{4}\Big)\leq 0.
			 \end{align*}
			Following the proof of Lemma \ref{L4.3} and summing from $ n=0 $ to $ n=M-1,$ in $ (\ref{3.12}) $, we get after   multiplying by $ k $ in the resulting inequality	
			\begin{align*}
				\norm{\hat{W}^{M}}^{2}&+k\sum_{n=0}^{M-1}\Big(\big(\frac{\theta^{2}\beta_{1}}{2}e^{-2\alpha k}-\frac{(1-e^{-2\alpha k})}{k} \big)\norm{\delta_{x}^{-}\hat{W}^{n+1}}_{h}^{2}+\big(\frac{\theta^{2}\beta_{2}}{2}e^{-2\alpha k}-\frac{(1-e^{-2\alpha k})}{k} \big)(\hat{W}_{0}^{n+1})^{2}
				\\&+\frac{\theta^{4}\beta_{3}}{2}e^{2\alpha t_{n}} (W_{0}^{n+1})^{4}+\big(\frac{\theta^{2}\beta_{4}}{2}e^{-2\alpha k}-\frac{(1-e^{-2\alpha k})}{k} \big)(\hat{W}_{N}^{n+1})^{2}+\frac{\theta^{4}\beta_{5}}{2}e^{2\alpha t_{n}}({W}_{N}^{n+1})^{4}\Big)
				\leq C \norm{W^{0}}_{1}^{2}.
			\end{align*}
			 Select $ k_{0} $ such that \eqref{3.15} holds for $ 0<k\leq k_{0}.$ This completes the proof after multiplying by $ e^{-2\alpha t_{M}} $.
		\end{proof}		
	\end{lemma}	
\begin{remark} Using summation by parts, we arrive at
	\label{r3.2}
	\begin{align*}
		h\sum_{i=1}^{N-1}\delta_{x}^{2}W_{i}^{n+\theta}\delta_{t}^{+}W_{i}^{n}&=\frac{1}{h}\sum_{i=1}^{N-1}\Big(\big(W_{i+1}^{n+\theta}-W_{i}^{n+\theta}\big)-\big(W_{i}^{n+\theta}-W_{i-1}^{n+\theta}\big)\Big)\delta_{t}^{+}W_{i}^{n},
		\\&=\sum_{i=1}^{N-1}\Big(\delta_{x}^{+}W_{i}^{n+\theta}-\delta_{x}^{-}W_{i}^{n+\theta} \Big)\delta_{t}^{+}W_{i}^{n},
		\\&=\sum_{i=2}^{N}\delta_{x}^{-}W_{i}^{n+\theta}\delta_{t}^{+}W_{i-1}^{n}-\sum_{i=1}^{N-1}\delta_{x}^{-}W_{i}^{n+\theta}\delta_{t}^{+}W_{i}^{n}.
	\end{align*}
Therefore, we can write
\begin{align*}
	h\sum_{i=1}^{N-1}\delta_{x}^{2}W_{i}^{n+\theta}\delta_{t}^{+}W_{i}^{n}=-\delta_{x}^{+}W_{0}^{n+\theta}\delta_{t}^{+}W_{0}^{n}-\Big(\delta_{t}^{+}(\delta_{x}^{-}W^{n}), \delta_{x}^{-}W^{n+\theta} \Big)_{h}+\delta_{x}^{-}W_{N}^{n+\theta}\delta_{t}^{+}W_{N}^{n}.
\end{align*}
\end{remark}
The following lemma deals with {\it{a priori}}  exponential bounds of the approximate solution in  the \(H^{1}\) and \(L^{\infty}\)-norms.
	\begin{lemma}
		\label{L3.6}
		Let \(W^{0}\in H^{1}(0,1)\) and \(\theta\in [\frac{1}{2}, 1]\). Then, there holds
		\begin{align*}
			\nu \norm{\delta_{x}^{-}W^{M}}_{h}^{2}&+ (c_{0}+w_{d})(W_{0}^{M})^{2}+ (c_{1}+w_{d})(W_{N}^{M})^{2}+\frac{1}{18c_{0}}(W_{0}^{M})^{4}+\frac{1}{18c_{1}}(W_{N}^{M})^{4}
			\\&+ Cke^{-2\alpha t_{M}}\sum_{n=0}^{M-1}e^{2\alpha t_{n}} \norm{\delta_{t}^{+}W^{n}}^{2}
			\leq C\norm{W^{0}}_{1}^{2}e^{-2\alpha t_{M}},
		\end{align*}
		and 
		\begin{align*}
			\norm{W^{n}}_{\infty}^{2}\leq  C\norm{W^{0}}_{1}^{2}e^{-2\alpha t_{M}}, \quad \forall \ 0\leq n\leq M.
		\end{align*}
	\end{lemma}
	\begin{proof}
		Multiplying the first, second, and third equation of \eqref{th} by \( \frac{h}{2} \delta_{t}^{+}W_{0}^{n}\), \(h\delta_{t}^{+}W_{i}^{n}\), and \(\frac{h}{2} \delta_{t}^{+}W_{N}^{n}\) respectively, and summing from \(i=1 \ \text{to} \  N-1\) to the second equation of \eqref{th}, we get after using summation by parts with Remark \ref{r3.2}
		\begin{align}
			\label{3.18}
			\norm{\delta_{t}^{+}W^{n}}^{2}&+ \nu \Big(\delta_{t}^{+}(\delta_{x}^{-}W^{n}), \delta_{x}^{-}W^{n+\theta}\Big)_{h}+ \nu g_{0}(W_{0}^{n+\theta}) \delta_{t}^{+}W_{0}^{n}-\nu g_{N}(W_{N}^{n+\theta})\delta_{t}^{+}W_{N}^{n}
			\nonumber\\&=-w_{d}\Big(\frac{h}{2}\delta_{x}^{+}W_{0}^{n+\theta}\delta_{t}^{+}W_{0}^{n}+ h\sum_{i=1}^{N-1}\delta_{x}^{c}W_{i}^{n+\theta}\delta_{t}^{+}W_{i}^{n}+\frac{h}{2}\delta_{x}^{-}W_{N}^{n+\theta}\delta_{t}^{+}W_{N}^{n} \Big)
			\nonumber\\&\quad -\frac{h}{2}\phi(W^{n+\theta}, W^{n+\theta})_{0}\delta_{t}^{+}W_{0}^{n}-h\sum_{i=1}^{N-1}\phi(W^{n+\theta}, W^{n+\theta})_{i}\delta_{t}^{+}W_{i}^{n}-\frac{h}{2}\phi(W^{n+\theta}, W^{n+\theta})_{N}\delta_{t}^{+}W_{N}^{n}.
		\end{align}
		Note that, using \((a+b)^{3}=a^{3}+b^{3}+3a^{2}b+3ab^{2}\)
		\begin{align*}
			\delta_{t}^{+}W_{0}^{n} (W_{0}^{n+\theta})^{3}&=\delta_{t}^{+}W_{0}^{n}\Big(k(\theta-\frac{1}{2})\delta_{t}^{+}W_{0}^{n}+ \frac{1}{2}(W_{0}^{n+1}+W_{0}^{n}) \Big)^{3},
			\\[1mm]&=\delta_{t}^{+}W_{0}^{n}\bigg(k^{3}\Big(\theta-\frac{1}{2}\Big)^{3}\big(\delta_{t}^{+}W_{0}^{n}\big)^{3}+\frac{1}{8}(W_{0}^{n+1}+W_{0}^{n})^{3}+\frac{3k^{2}}{2}\big(\theta-\frac{1}{2}\big)^{2}(\delta_{t}^{+}W_{0}^{n})^{2}(W_{0}^{n+1}+W_{0}^{n})
			\\&\quad+\frac{3k}{4}\Big(\theta-\frac{1}{2}\Big)(\delta_{t}^{+}W_{0}^{n})\Big(W_{0}^{n+1}+W_{0}^{n}\Big)^{2} \bigg),
			\\[1mm]&=k^{3}\Big(\theta-\frac{1}{2}\Big)^{3}(\delta_{t}^{+}W_{0}^{n})^{4}+ \frac{3k^{2}}{2}\Big(\theta-\frac{1}{2}\Big)^{2}(\delta_{t}^{+}W_{0}^{n})^{3}\Big(W_{0}^{n+1}+W_{0}^{n}\Big)
			\\[1mm]& \quad +\frac{3k}{4}\Big(\theta-\frac{1}{2}\Big)(\delta_{t}^{+}W_{0}^{n})^{2}\Big(W_{0}^{n+1}+W_{0}^{n}\Big)^{2}+\frac{1}{8}\delta_{t}^{+}W_{0}^{n}\Big(W_{0}^{n+1}+W_{0}^{n}\Big)^{3},
		\end{align*}
		where \(W_{0}^{n+\theta}=(\theta-\frac{1}{2})(W_{0}^{n+1}-W_{0}^{n})+ \frac{1}{2}(W_{0}^{n+1}+W_{0}^{n})=k(\theta-\frac{1}{2})\delta_{t}^{+}W_{0}^{n}+ \frac{1}{2}(W_{0}^{n+1}+W_{0}^{n})\).
		
		Also 
		\begin{align*}
			\frac{1}{8}\delta_{t}^{+}W_{0}^{n}(W_{0}^{n+1}+W_{0}^{n})^{3}&=\frac{1}{8}\delta_{t}^{+}(W_{0}^{n})^{4}+\frac{3}{8k}\Big(W_{0}^{n}(W_{0}^{n+1})^{3}-W_{0}^{n+1}(W_{0}^{n})^{3}\Big),
			\\& = \frac{1}{8}\delta_{t}^{+}(W_{0}^{n})^{4}+\frac{3}{8}\delta_{t}^{+}(W_{0}^{n})^{2}W_{0}^{n+1}W_{0}^{n}.
		\end{align*}
		Using the  Young's inequality with Lemma \ref{L4.4}, we have
		\begin{align}
			\label{3.19}
			\nu \Big(\delta_{t}^{+}(\delta_{x}^{-}W^{n}), \delta_{x}^{-}W^{n+\theta}\Big)_{h}&+(c_{0}+w_{d})\delta_{t}^{+}W_{0}^{n} W_{0}^{n+\theta}+ (c_{1}+w_{d})\delta_{t}^{+}W_{N}^{n} W_{N}^{n+\theta}
			\nonumber\\&+ \frac{k^{3}}{18c_{0}}\big(\theta-\frac{1}{2}\big)^{3}(\delta_{t}^{+}W_{0}^{n})^{4}+ \frac{1}{36c_{0}}\delta_{t}^{+}(W_{0}^{n})^{4}
			\nonumber\\& +\frac{k^{3}}{18c_{1}}\big(\theta-\frac{1}{2}\big)^{3}(\delta_{t}^{+}W_{N}^{n})^{4}+ \frac{1}{36c_{1}}\delta_{t}^{+}(W_{N}^{n})^{4}
			+\frac{1}{2}\norm{\delta_{t}^{+}W^{n}}^{2}
			\nonumber\\&\leq w_{d}^{2}\norm{\delta_{x}^{-}W^{n+\theta}}_{h}^{2}+ \frac{19}{18}\norm{W^{n+\theta}}_{\infty}^{2}\norm{\delta_{x}^{-}W^{n+\theta}}_{h}^{2}+ \frac{C}{k}\norm{W^{n}}_{\infty}^{2}(W_{0}^{n+1})^{2}
			\nonumber\\&\quad + \frac{C}{k}\Big((W_{0}^{n+1})^{4}+ (W_{0}^{n})^{4}\Big)+\frac{C}{k}\norm{W^{n}}_{\infty}^{2}(W_{N}^{n+1})^{2}
			+ \frac{C}{k}\Big((W_{N}^{n+1})^{4}+ (W_{N}^{n})^{4}\Big).
		\end{align}
		Since \[\Big(\delta_{t}^{+}(\delta_{x}^{-}W^{n}), \delta_{x}^{-}W^{n+\theta}\Big)_{h}=\frac{1}{2}\delta_{t}^{+}\norm{\delta_{x}^{-}W^{n}}_{h}^{2}+ k(\theta-\frac{1}{2})\norm{\delta_{t}^{+}(\delta_{x}^{-}W^{n})}_{h}^{2},
		\]
		and
		\[\delta_{t}^{+}W_{0}^{n} W_{0}^{n+\theta}=\frac{1}{2}\delta_{t}^{+}(W_{0}^{n})^{2}+ k(\theta-\frac{1}{2})(\delta_{t}^{+}W_{0}^{n}),
		\]
		then, for the case \(\theta\geq \frac{1}{2}\), we obtain from \eqref{3.19}
		\begin{align*}
			\nu\delta_{t}^{+}\norm{\delta_{x}^{-}W^{n}}_{h}^{2}&+(c_{0}+w_{d})\delta_{t}^{+}(W_{0}^{n})^{2}+(c_{1}+w_{d})\delta_{t}^{+}(W_{N}^{n})^{2}\\&+\frac{1}{18c_{0}}\delta_{t}^{+}(W_{0}^{n})^{4}+\frac{1}{18c_{1}}\delta_{t}^{+}(W_{N}^{n})^{4}+\norm{\delta_{t}^{+}W^{n}}^{2}
			\\&\leq 2w_{d}^{2}\norm{\delta_{x}^{-}W^{n+\theta}}_{h}^{2}+ \frac{19}{9}\norm{W^{n+\theta}}_{\infty}^{2}\norm{\delta_{x}^{-}W^{n+\theta}}_{h}^{2}+ \frac{C}{k}\norm{W^{n}}_{\infty}^{2}(W_{0}^{n+1})^{2}
			\nonumber\\&\quad + \frac{C}{k}\Big((W_{0}^{n+1})^{4}+ (W_{0}^{n})^{4}\Big)+\frac{C}{k}\norm{W^{n}}_{\infty}^{2}(W_{N}^{n+1})^{2}
			+ \frac{C}{k}\Big((W_{N}^{n+1})^{4}+ (W_{N}^{n})^{4}\Big).
		\end{align*}
	Multiplying by $ e^{2\alpha t_{n+1}} $ to the above equation and using \eqref{exp}, we get after multiplying by $ e^{-2\alpha k} $ in the resulting inequality
	\begin{align*}
	\Big(\nu\delta_{t}^{+}\norm{\delta_{x}^{-}\hat{W}^{n}}_{h}^{2} &+(c_{0}+w_{d})\delta_{t}^{+}(\hat{W}_{0}^{n})^{2}+(c_{1}+w_{d})\delta_{t}^{+}(\hat{W}_{N}^{n})^{2}\Big)
	\\&+e^{2\alpha t_{n+1}}\Big(\frac{1}{18c_{0}}\delta_{t}^{+}({W}_{0}^{n})^{4}+\frac{1}{18c_{1}}\delta_{t}^{+}({W}_{N}^{n})^{4} \Big) +e^{2\alpha t_{n}}\norm{\delta_{t}^{+}W^{n}}^{2}\\&\leq C(\alpha)\Big( \norm{\delta_{x}^{-}\hat{W}^{n+1}}_{h}^{2}+(\hat{W}_{0}^{n+1})^{2}+ (\hat{W}_{N}^{n+1})^{2}+e^{2\alpha t_{n}}\big(({W}_{0}^{n+1})^{4}+({W}_{N}^{n+1})^{4}\big)  \Big)
	\\&\quad +Ce^{2\alpha t_{n}} \norm{\delta_{x}^{-}W^{n+\theta}}_{h}^{2}+ \frac{19e^{2\alpha t_{n}}}{9}\norm{W^{n+\theta}}_{\infty}^{2}\norm{\delta_{x}^{-}W^{n+\theta}}_{h}^{2}
	\nonumber\\&\quad+ \frac{Ce^{2\alpha t_{n}}}{k}\norm{W^{n}}_{\infty}^{2}(W_{0}^{n+1})^{2}
	 + \frac{Ce^{2\alpha t_{n}}}{k}\Big((W_{0}^{n+1})^{4}+ (W_{0}^{n})^{4}\Big)
	 \nonumber\\&\quad+\frac{Ce^{2\alpha t_{n}}}{k}\norm{W^{n}}_{\infty}^{2}(W_{N}^{n+1})^{2}
	+ \frac{Ce^{2\alpha t_{n}}}{k}\Big((W_{N}^{n+1})^{4}+ (W_{N}^{n})^{4}\Big).
	\end{align*}	
	Multiplying by $ k $ and summing over \(n=0\) to \(n=M-1\) to the above inequality. Then, using discrete Gronwall's inequality with Lemma \ref{L4.3}, we have after multiplying by $ e^{-2\alpha t_{M}} $ in the resulting inequality
		\begin{align*}
			\nu \norm{\delta_{x}^{-}W^{M}}_{h}^{2}&+ (c_{0}+w_{d})(W_{0}^{M})^{2}+ (c_{1}+w_{d})(W_{N}^{M})^{2}+\frac{1}{18c_{0}}(W_{0}^{M})^{4}+\frac{1}{18c_{1}}(W_{N}^{M})^{4}
			\\&+ Cke^{-2\alpha t_{M}}\sum_{n=0}^{M-1}e^{2\alpha t_{n}} \norm{\delta_{t}^{+}W^{n}}^{2}
			\leq C\norm{W^{0}}_{1}^{2}e^{-2\alpha t_{M}},
		\end{align*}
where we choose \(k\) sufficiently small such that \((\nu-Ck)>0\).
		
		This completes the first part of the proof.
		
		Now
		\[
		\norm{W^{n}}_{\infty}^{2}\leq 2 \Big(\norm{W^{n+\theta}}^{2}+ W_{j}^{n+\theta}\Big), \quad \forall \ n\geq 0, \ \text{and} \ j=0, N.
		\]
		Using the first part, we arrive at
		\begin{align*}
			\norm{W^{n}}_{\infty}^{2}\leq C\norm{W^{0}}_{1}^{2}e^{-2\alpha t_{M}}.
		\end{align*}
		The proof is completed.
	\end{proof}

\begin{remark}
	We deliberately omitted the \( H^1 \)-stability analysis for \( \theta < \frac{1}{2} \) because the resulting time-step restriction is not practically useful. It behaves like an explicit scheme, requiring extremely small time steps for fine spatial meshes. Moreover, the stability condition in Lemma \ref{L4.5} involves nonlinear, solution-dependent terms on the right-hand side, which cannot be controlled a priori. This makes the analysis both technically cumbersome and not particularly helpful for designing robust algorithms \cite{MR4242164}. Hence we focus only on $\theta \in [\frac{1}{2}, 1]$.
		
\end{remark}
	\section{Error analysis.}
	This section contains the error analysis of the finite difference scheme for  $\theta \in [\frac{1}{2}, 1]$. Assume that the solution $ w \in C^{4,3}([0, 1], [0, T])$  of (\ref{w11}) - (\ref{w12}) . Using Taylor series expansion, we obtain, for all $ 0 \leq n \leq M-1 $
	\begin{align}
		\label{eq5.1}
		\begin{cases}
			\delta_{t}^{+}w_{0}^{n}-\frac{2\nu}{h^{2}}\left[w_{1}^{n+\theta}-w_{0}^{n+\theta}-hg_{0}(w_{0}^{n+\theta})\right]+w_{d}\delta_{x}^{+}w_{0}^{n+\theta}+\phi(w^{n+\theta},w^{n+\theta})_{0}=T_{0}^{n+\theta} ,  &  i=0 ,  \vspace{0.5cm} \\ 
			\vspace{0.5cm}
			\delta_{t}^{+} w_{i}^{n}-\nu \delta_{x}^{2} w_{i}^{n+\theta}+ w_{d}\delta_{x}^{c} w_{i}^{n+\theta}+ \phi(w^{n+\theta},w^{n+\theta})_{i} =T_{i}^{n+\theta},\medspace \ \  i=1,2,\ldots, N-1  , \\
			\delta_{t}^{+}w_{N}^{n}-\frac{2\nu}{h^{2}}[w_{N-1}^{n+\theta}-w_{N}^{n+\theta}+hg_{N}(w_{N}^{n+\theta})]+w_{d}\delta_{x}^{-}w_{N}^{n+\theta}+\phi(w^{n+\theta},w^{n+\theta})_{N}=T_{N}^{n+\theta}, & i=N,    \vspace{0.5cm} \\
			w_{i}^{0}=w_0(x_{i}), \medspace  \ i=0, 1,\ldots, N,\\
		\end{cases}		
	\end{align} 
	where \begin{align}
		\label{eq5.2}
		T_{i}^{n+\theta}=
		\begin{cases}
			O(h+k^{2}+(\theta-\frac{1}{2})k), &  i=0 ,  \vspace{0.5cm} \\ 
			\vspace{0.5cm}	
			O(h^{2}+k^{2}+(\theta-\frac{1}{2})k), &  i=1,2,\ldots,N-1 , \\
			O(h+k^{2}+(\theta-\frac{1}{2})k), &  i=N, \\
		\end{cases}	
	\end{align}
	and  $ g_{0}(w_{0}^{n})=v_{0}(w(x_{0}, t^{n}))= v_{0}(t^{n})$ and $ g_{N}(w_{N}^{n})=v_{1}(w(x_{N}, t^{n}))= v_{1}(t^{n}).$
	
	We define the error at the $ n^{th} $ time level as
	\begin{align*}
		e_{i}^{n}:=w_{i}^{n}-W_{i}^{n}, \medspace \ \text{for} \ i=0,1,\ldots,N. 
	\end{align*}
	
	Subtracting  (\ref{th})  from (\ref{eq5.1}) for all $ 0 \leq n \leq M-1 $,  we get
	\begin{align}
		\label{eq5.3}
		\begin{cases}
			\delta_{t}^{+}e_{0}^{n}-\frac{2\nu}{h^{2}}\left[e_{1}^{n+\theta}-e_{0}^{n+\theta}-h\left(g_{0}(w_{0}^{n+\theta})-g_{0}(W_{0}^{n+\theta})\right)\right]+w_{d}\delta_{x}^{+}e_{0}^{n+\theta}\\+\left(\phi(w^{n+\theta},w^{n+\theta})_{0}-\phi(W^{n+\theta},W^{n+\theta})_{0}\right)=T_{0}^{n+\theta} ,  \medspace \  i=0   , \vspace{0.5cm} \\ 
			\vspace{0.5cm}
			\delta_{t}^{+} e_{i}^{n}-\nu \delta_{x}^{2} e_{i}^{n+\theta}+ w_{d}\delta_{x}^{c} e_{i}^{n+\theta}+ \left(\phi(w^{n+\theta},w^{n+\theta})_{i}-\phi(W^{n+\theta},W^{n+\theta})_{i}\right) =T_{i}^{n+\theta}, \medspace \  1\leq i\leq N-1  , \\
			\delta_{t}^{+}e_{N}^{n}-\frac{2\nu}{h^{2}}\left[e_{N-1}^{n+\theta}-e_{N}^{n+\theta}+h\left(g_{N}(w_{N}^{n+\theta})-g_{N}(W_{N}^{n+\theta})\right)\right]+w_{d}\delta_{x}^{-}e_{N}^{n+\theta}\\ +\left(\phi(w^{n+\theta},w^{n+\theta})_{N}-\phi(W^{n+\theta},W^{n+\theta})_{N}\right)=T_{N}^{n+\theta} , \medspace \  i=N  , \vspace{0.5cm} \\
			e_{i}^{0}=0, \quad  \ i=0,1,\ldots,N.\\
		\end{cases}		
	\end{align}

	\begin{theorem}
		\label{5.1}
		Let $ w \in C^{4,3}([0, 1], [0, T])$ be solution of (\ref{w11}) - (\ref{w12}) and $ W_{i}^{n} $ be the solution of (\ref{th}). Suppose that $ 0\leq \alpha \leq \frac{\theta^{2}}{2}\min\{\nu, \frac{c_{0}+w_{d}-2C}{2}, \frac{c_{1}+3w_{d}-2C}{2} \} $, where $ \theta\in \left[\frac{1}{2}, 1 \right] $. Let $ k_{0}>0 $ such that for $ 0<k\leq k_{0} $
		\begin{align}
			\label{4.41}
		 e^{2\alpha k}\leq 1+k \theta^{2}\min\{\nu, \frac{c_{0}+w_{d}-2C}{2}, \frac{c_{1}+3w_{d}-2C}{2}\}.
		\end{align}	
		 For $ \theta\in \left[\frac{1}{2}, 1 \right] $,  there exists  a positive constant  $C$ independent of $ h $ and $ k $ such that for all $ 0\leq n\leq M-1 $
	\begin{align*}
		\norm{{e}^{M}}^{2}+ e^{-2\alpha t_{M}}k\beta^{*1} \sum_{n=0}^{M-1}\Big(\norm{\delta_{x}^{-}\hat{e}^{n+1}}_{h}^{2}&+(\hat{e}_{0}^{n+1})^{2}+(\hat{e}_{N}^{n+1})^{2}  \Big)
		+ e^{-2\alpha t_{M}}k\theta^{4}\sum_{n=0}^{M-1} e^{2\alpha t_{n}}\Big(\frac{1}{36c_{0}}(e_{0}^{n+1})^{4}
		\\&+\frac{1}{36c_{1}}(e_{N}^{n+1})^{4}\Big)
		\leq Ce^{-2\alpha t_{M}} \norm{W^{0}}_{1}^{2}\max_{0\leq n\leq M-1}\norm{T^{n+\theta}}^{2},
	\end{align*}
		where $ T_{i}^{n+\theta}$ is given in (\ref{eq5.2}) and $\beta^{*1}  $ is given in \eqref{4.12} below.
	\end{theorem}
	\begin{proof}
		Multiplying the first, second, and third equations of (\ref{eq5.3}) by $ \frac{h}{2} e_{0}^{n+\theta} ,$ $ h e_{i}^{n+\theta} $, and $ \frac{h}{2}e_{N}^{n+\theta} $  respectively. Summing from $ i=1$ to $ N-1 $ in the second equation of  (\ref{eq5.3}),    using summation by parts, we obtain
		\begin{align}
			\label{4.6}
			\nonumber	\left( \delta_{t}^{+}e^{n}, e^{n+\theta}\right) &+\nu \norm{\delta_{x}^{-}e^{n+\theta}}_{h}^{2}+\nu g_{0}(e_{0}^{n+\theta})e_{0}^{n+\theta}- \nu g_{N}(e_{N}^{n+\theta})e_{N}^{n+\theta}\nonumber\\& + w_{d}\left(\frac{h}{2}\delta_{x}^{+}e_{0}^{n+\theta}e_{0}^{n+\theta} +h \sum_{i=1}^{N-1}\delta_{x}^{c}e_{i}^{n+\theta}e_{i}^{n+\theta}+ \frac{h}{2}\delta_{x}^{-}e_{N}^{n+\theta}e_{N}^{n+\theta}\right)\nonumber\\&=-\frac{2}{3c_{0}\nu}(e_{0}^{n+\theta})^{2}W_{0}^{n+\theta}w_{0}^{n+\theta}+ \frac{2}{3c_{1}\nu}(e_{N}^{n+\theta})^{2}W_{N}^{n+\theta}w_{N}^{n+\theta}+(T^{n+\theta}, e^{n+\theta})\nonumber \\& -\left(\phi(w^{n+\theta},w^{n+\theta})-\phi(W^{n+\theta},W^{n+\theta}),e^{n+\theta}\right).
		\end{align}
		Setting  $ w_{i}^{n+\theta}=e_{i}^{n+\theta}+W_{i}^{n+\theta},$
		to the right hand side of the above equation, we can write
		\begin{align}
			\label{N}
			\nonumber	\left( \delta_{t}^{+}e^{n}, e^{n+\theta}\right) &+\nu \norm{\delta_{x}^{-}e^{n+\theta}}_{h}^{2}+\nu g_{0}(e_{0}^{n+\theta})e_{0}^{n+\theta}- \nu g_{N}(e_{N}^{n+\theta})e_{N}^{n+\theta}\nonumber\\& + w_{d}\left(\frac{h}{2}\delta_{x}^{+}e_{0}^{n+\theta}e_{0}^{n+\theta} +h \sum_{i=1}^{N-1}\delta_{x}^{c}e_{i}^{n+\theta}e_{i}^{n+\theta}+ \frac{h}{2}\delta_{x}^{-}e_{N}^{n+\theta}e_{N}^{n+\theta}\right)\nonumber\\&=-\frac{2}{3c_{0}\nu}(e_{0}^{n+\theta})^{2}W_{0}^{n+\theta}(e_{0}^{n+\theta}+W_{0}^{n+\theta})+ \frac{2}{3c_{1}\nu}(e_{N}^{n+\theta})^{2}W_{N}^{n+\theta}(e_{N}^{n+\theta}+W_{N}^{n+\theta})\nonumber\\&+(T^{n+\theta}, e^{n+\theta})\nonumber  -\left(\phi(e^{n+\theta},e^{n+\theta}),e^{n+\theta}\right)-\left(\phi(e^{n+\theta},W^{n+\theta}),e^{n+\theta}\right)\nonumber\\&-\left(\phi(W^{n+\theta},e^{n+\theta}),e^{n+\theta}\right).
		\end{align}
		
		Using the Cauchy-Schwarz inequality and Young's inequality with Lemmas \ref{L4.1} and \ref{L4.2} to (\ref{N}), we obtain
		\begin{align}
			\label{4.7}
			\nonumber	\left( \delta_{t}^{+}e^{n},e^{n+\theta}\right) &+\nu \norm{\delta_{x}^{-}e^{n+\theta}}_{h}^{2}+\left((c_{0}+w_{d})(e_{0}^{n+\theta})^{2}+\frac{2}{9c_{0}}(e_{0}^{n+\theta})^{4}+(c_{1}+w_{d})(e_{N}^{n+\theta})^{2}+\frac{2}{9c_{1}}(e_{N}^{n+\theta})^{4}\right)\\&+\frac{w_{d}}{2}\left((e_{N}^{n+\theta})^{2}-(e_{0}^{n+\theta})^{2}\right)
			\nonumber\\&\leq C\norm{W^{n+\theta}}_{1}^{2}\norm{e^{n+\theta}}^{2}+\frac{1}{3}\left((e_{0}^{n+\theta})^{3}-(e_{N}^{n+\theta})^{3}\right)+\norm{T^{n+\theta}}\norm{e^{n+\theta}}\nonumber\\&\quad+\frac{1}{18c_{0}}(e_{0}^{n+\theta})^{4}+\frac{1}{18c_{1}}(e_{N}^{n+\theta})^{4}+C(e_{0}^{n+\theta})^{2}+C(e_{N}^{n+\theta})^{2}.
		\end{align}
		A use of the Young's inequality, the second term on the right-hand side of $(\ref{4.7})$ to give
		\begin{align}
			\label{4.8}
			\frac{1}{3}(e_{0}^{n+\theta})^{3}\leq \frac{c_{0}}{2}(e_{0}^{n+\theta})^{2}+\frac{1}{18c_{0}}(e_{0}^{n+\theta})^{4},\\
			\label{4.9}
			\frac{1}{3}(e_{N}^{n+\theta})^{3}\leq \frac{c_{1}}{2}(e_{N}^{n+\theta})^{2}+\frac{1}{18c_{1}}(e_{N}^{n+\theta})^{4}.
		\end{align}
		Note that
		\begin{align}
			\label{4.10}
			\left( \delta_{t}^{+}e^{n},     e^{n+\theta}\right)=\frac{1}{2}\delta_{t}^{+}\norm{e^{n}}^{2}+k(\theta-\frac{1}{2})\norm{\delta_{t}^{+}e^{n}}^{2}.
		\end{align}
		Substituting $(\ref{4.8})-(\ref{4.10})$ into  $(\ref{4.7})$ and using  the Young's inequality, we arrive at
		\begin{align}
			\label{4.11}	
			\nonumber \frac{1}{2}\delta_{t}^{+}\norm{e^{n}}^{2}&+k(\theta-\frac{1}{2})\norm{\delta_{t}^{+}e^{n}}^{2}+\nu\norm{\delta_{x}^{-}e^{n+\theta}}_{h}^{2} \\& \nonumber  +\left((\frac{c_{0}}{2}+\frac{w_{d}}{2}-C)(e_{0}^{n+\theta})^{2}+\frac{1}{9c_{0}}(e_{0}^{n+\theta})^{4}+(\frac{c_{1}}{2}+\frac{3w_{d}}{2}-C)(e_{N}^{n+\theta})^{2}+\frac{1}{9c_{1}}(e_{N}^{n+\theta})^{4}\right)\\&\leq C\norm{W^{n+\theta}}_{1}^{2}\norm{e^{n+\theta}}^{2}+ \frac{1}{2}\norm{T^{n+\theta}}^{2}.
		\end{align}
		For the case $ \theta\geq \frac{1}{2} ,$ we deduce from \eqref{4.11}
		\begin{align*}	
			\nonumber \frac{1}{2}\delta_{t}^{+}\norm{e^{n}}^{2}&+\nu\norm{\delta_{x}^{-}e^{n+\theta}}_{h}^{2} +\Big((\frac{c_{0}}{2}+\frac{w_{d}}{2}-C)(e_{0}^{n+\theta})^{2}+\frac{1}{9c_{0}}(e_{0}^{n+\theta})^{4}
			 \\& \nonumber +(\frac{c_{1}}{2}+\frac{3w_{d}}{2}-C)(e_{N}^{n+\theta})^{2}+\frac{1}{9c_{1}}(e_{N}^{n+\theta})^{4}\Big)\leq C\norm{W^{n+\theta}}_{1}^{2}\norm{e^{n+\theta}}^{2}+ \frac{1}{2}\norm{T^{n+\theta}}^{2}.
		\end{align*}
	Multiplying by $ 2e^{2\alpha t_{n+1}} $ and using \eqref{exp}, we arrive at
	\begin{align*}
		e^{2\alpha k}\delta_{t}^{+}\norm{\hat{e}^{n}}^{2}&-\frac{(e^{2\alpha k}-1)}{k}\norm{\hat{e}^{n+1}}^{2} +2\nu e^{2\alpha t_{n+1}}\norm{\delta_{x}^{-}e^{n+\theta}}_{h}^{2}
		  +e^{2\alpha t_{n+1}}\Big((c_{0}+w_{d}-2C)(e_{0}^{n+\theta})^{2}
		  \\& +\frac{1}{18c_{0}}(e_{0}^{n+\theta})^{4}+(c_{1}+3w_{d}-2C)(e_{N}^{n+\theta})^{2}+\frac{1}{18c_{1}}(e_{N}^{n+\theta})^{4}\Big)
		   \\& \leq Ce^{2\alpha t_{n+1}}\norm{W^{n+\theta}}_{1}^{2}\norm{e^{n+\theta}}^{2}+ e^{2\alpha t_{n+1}}\norm{T^{n+\theta}}^{2},
	\end{align*}
where $ \hat{e}^{n}=e^{\alpha t_{n}} e^{n}. $

Following the proof of Lemma \ref{L4.3} with Young's inequality and multiplying by $ e^{-2\alpha k} $ in the resulting inequality yields
\begin{align*}
	\delta_{t}^{+}\norm{\hat{e}^{n}}^{2}&-\frac{(1-e^{-2\alpha k})}{k}\norm{\hat{e}^{n+1}}^{2}+e^{-2\alpha k}\theta^{2}\nu \norm{\delta_{x}^{-}\hat{e}^{n+1}}_{h}^{2}+ e^{-2\alpha k}\theta^{2}\Big(\frac{(c_{0}+w_{d}-2C)}{2}(\hat{e}_{0}^{n+1})^{2}
	\\&+ \frac{(c_{1}+3w_{d}-2C)}{2}(\hat{e}_{N}^{n+1})^{2} \Big)+ \theta^{4}e^{2\alpha t_{n}}\Big(\frac{1}{36c_{0}}(e_{0}^{n+1})^{4}+\frac{1}{36c_{1}}(e_{N}^{n+1})^{4}\Big)
	\\&\leq (1-\theta)^{2}e^{2\alpha t_{n}}\Big( 2\nu\norm{\delta_{x}^{-}{e}^{n}}_{h}^{2}+(c_{0}+w_{d}-2C)({e}_{0}^{n})^{2}+(c_{1}+3w_{d}-2C)({e}_{N}^{n})^{2} 
	\\&\quad+ C(1-\theta)^{2}\big((e_{0}^{n})^{4}+(e_{N}^{n})^{4}\big) \Big)+  Ce^{2\alpha t_{n}}\norm{W^{n+\theta}}_{1}^{2}\norm{e^{n+\theta}}^{2}+ e^{2\alpha t_{n}}\norm{T^{n+\theta}}^{2}.
\end{align*}
	
	Using discrete Poincar\'e inequality and summing  from \(n=0\) to \(M-1\) gives	
		\begin{align*}
			\norm{\hat{e}^{M}}^{2}+ k\beta^{*1} \sum_{n=0}^{M-1}\Big(\norm{\delta_{x}^{-}\hat{e}^{n+1}}_{h}^{2}+(\hat{e}_{0}^{n+1})^{2}&+(\hat{e}_{N}^{n+1})^{2}  \Big)
			+ k\theta^{4}\sum_{n=0}^{M-1} e^{2\alpha t_{n}}\Big(\frac{1}{36c_{0}}(e_{0}^{n+1})^{4}+\frac{1}{36c_{1}}(e_{N}^{n+1})^{4}\Big)
	\\&	\leq Ck\theta^{2}\sum_{n=0}^{M}e^{2\alpha t_{n+1}}\norm{W^{n+1}}_{1}^{2}\norm{e^{n}}^{2}+k\sum_{n=0}^{M-1}e^{2\alpha t_{n}}\norm{T^{n+\theta}}^{2},
		\end{align*}
	where \begin{align}
		\label{4.12}
		0<\beta^{*1}=\min\{ \Big(e^{-2\alpha k}\theta^{2}\nu&-\frac{(1-e^{-2\alpha k})}{k}\Big), \Big(e^{-2\alpha k}\frac{\theta^{2}(c_{0}+w_{d}-2C)}{2}-\frac{(1-e^{-2\alpha k})}{k}\Big), 
		\nonumber\\&\quad  \Big(e^{-2\alpha k}\frac{\theta^{2}(c_{1}+3w_{d}-2C)}{2}-\frac{(1-e^{-2\alpha k})}{k}\Big)  \},
	\end{align}
and we choose \(c_{0}, c_{1}\), and \(w_{d}\) such that \({c_{0}}+{w_{d}}-C\) and \({c_{1}}+3{w_{d}}-C\) are nonzero. 	Select $ k_{0} $ such that \eqref{4.41} is satisfied for $ 0<k\leq k_{0}$.
	
	Suppose that \(k\) is sufficiently small such that \( (1- Ck\theta^{2})>0.\) Then, using the discrete Gronwall's inequality with Lemmas \ref{L4.3} and \ref{L3.6}, it follows that
		\begin{align*}
			\norm{{e}^{M}}^{2}+ e^{-2\alpha t_{M}}k\beta^{*1} \sum_{n=0}^{M-1}\Big(\norm{\delta_{x}^{-}\hat{e}^{n+1}}_{h}^{2}&+(\hat{e}_{0}^{n+1})^{2}+(\hat{e}_{N}^{n+1})^{2}  \Big)
			+ e^{-2\alpha t_{M}}k\theta^{4}\sum_{n=0}^{M-1} e^{2\alpha t_{n}}\Big(\frac{1}{36c_{0}}(e_{0}^{n+1})^{4}
			\\&+\frac{1}{36c_{1}}(e_{N}^{n+1})^{4}\Big)
			\leq Ce^{-2\alpha t_{M}} \norm{W^{0}}_{1}^{2}\max_{0\leq n\leq M-1}\norm{T^{n+\theta}}^{2},
		\end{align*}	
	where $ e^{-2\alpha t_{M}}k\sum_{n=0}^{M}e^{2\alpha t_{n}}\leq C. $	
		The proof is completed.
	\end{proof}

	The following lemma presents the error estimate of the state variable in the \(H^{1}\) and \(L^{\infty}\)-norms.
	\begin{lemma}
		Suppose that \(\theta\in [\frac{1}{2}, 1]\). Then, the following holds
		\begin{align*}
			\nu \norm{\delta_{x}^{-}e^{M}}_{h}^{2}&+ (c_{0}+w_{d})(e_{0}^{M})^{2}+(c_{1}+w_{d})(e_{N}^{M})^{2}
			+\frac{1}{18c_{0}}(e_{0}^{M})^{4}+\frac{1}{18c_{1}}(e_{N}^{M})^{4}
			\\&+ke^{-2\alpha t_{M}}\sum_{n=0}^{M-1}e^{2\alpha t_{n}}\norm{\delta_{t}^{+}e^{n}}^{2}
			\leq C(\norm{W^{0}}_{1} )e^{-2\alpha t_{M}}\max_{0\leq n\leq M-1}\norm{T^{n+\theta}}^{2},
		\end{align*}
		and 
		\begin{align*}
			\norm{e^{n}}_{\infty}\leq C(\norm{W^{0}}_{1} )e^{-2\alpha t_{M}}(h+(\theta-\frac{1}{2})k+ k^{2}), \quad 1\leq n\leq M.
		\end{align*}
	\end{lemma}
	\begin{proof}
		Multiplying the first, second, and third equation of \eqref{th} by \( \frac{h}{2} \delta_{t}^{+}e_{0}^{n}\), \(h\delta_{t}^{+}e_{i}^{n}\), and \(\frac{h}{2} \delta_{t}^{+}e_{N}^{n}\) respectively, and summing over \(i=1 \ \text{to} \  N-1\) in the second equation of \eqref{eq5.3}, we get 
		\begin{align}
			\label{4.17}
			\norm{\delta_{t}^{+}e^{n}}^{2}&+\nu \Big(\delta_{t}^{+}(\delta_{x}^{-}e^{n}),\delta_{x}^{-}e^{n+\theta} \Big)_{h}+\nu g_{0}(e_{0}^{n+\theta})\delta_{t}^{+}e_{0}^{n}-\nu g_{1}(e_{N}^{n+\theta})\delta_{t}^{+}e_{N}^{n}
			\nonumber\\&=-\frac{2}{3c_{0}\nu}e_{0}^{n+\theta}W_{0}^{n+\theta}(e_{0}^{n+\theta}+W_{0}^{n+\theta})\delta_{t}^{+}e_{0}^{n}+ \frac{2}{3c_{1}\nu}e_{N}^{n+\theta}W_{N}^{n+\theta}(e_{N}^{n+\theta}+W_{N}^{n+\theta})\delta_{t}^{+}e_{N}^{n}
			\nonumber\\&\quad -w_{d}\left(\frac{h}{2}\delta_{x}^{+}e_{0}^{n+\theta}\delta_{t}^{+}e_{0}^{n} +h \sum_{i=1}^{N-1}\delta_{x}^{c}e_{i}^{n+\theta}\delta_{t}^{+}e_{i}^{n}+ \frac{h}{2}\delta_{x}^{-}e_{N}^{n+\theta}\delta_{t}^{+}e_{N}^{n}\right)
			\nonumber\\&\quad -\left(\phi(e^{n+\theta},e^{n+\theta}),\delta_{t}^{+}e^{n}\right)-\left(\phi(e^{n+\theta},W^{n+\theta}),\delta_{t}^{+}e^{n}\right)\nonumber\\&\quad -\left(\phi(W^{n+\theta},e^{n+\theta}),\delta_{t}^{+}e^{n}\right)+(T^{n+\theta}, \delta_{t}^{+}e^{n}).
		\end{align}
		
		On the right hand side of \eqref{4.17}, the first term is bounded by 
		\begin{align*}
			-\frac{2}{3c_{0}\nu}e_{0}^{n+\theta}W_{0}^{n+\theta}(e_{0}^{n+\theta}&+W_{0}^{n+\theta})\delta_{t}^{+}e_{0}^{n}\\&\leq C\Big(\frac{(2\theta-1)^{2}}{k}+\frac{1}{k}\Big)(W_{0}^{n+\theta})^{2}(e_{0}^{n+1})^{2}+\frac{C}{k}(W_{0}^{n+\theta})^{2}(e_{0}^{n+1})^{4}\\&\quad +\frac{C}{k}(e_{0}^{n})^{4}+\frac{C}{k}(e_{0}^{n})^{2},
		\end{align*}
		and 
		\begin{align*}
			\frac{2}{3c_{1}\nu}e_{N}^{n+\theta}W_{N}^{n+\theta}(e_{N}^{n+\theta}&+W_{N}^{n+\theta})\delta_{t}^{+}e_{N}^{n}\\&\leq C\Big(\frac{(2\theta-1)^{2}}{k}+\frac{1}{k}\Big)(W_{N}^{n+\theta})^{2}(e_{N}^{n+1})^{2}+\frac{C}{k}(W_{N}^{n+\theta})^{2}(e_{N}^{n+1})^{4}\\&\quad +\frac{C}{k}(e_{N}^{n})^{4}+C\big(\frac{1}{k}+\frac{(W_{N}^{n+\theta})^{2}}{k}\big)(e_{N}^{n})^{2}.
		\end{align*}
		
		By Young's inequality and Lemma	\ref{L4.4}, the second term on the right hand side of \eqref{4.17} yields
		\begin{align*}
			-w_{d}\left(\frac{h}{2}\delta_{x}^{+}e_{0}^{n+\theta}\delta_{t}^{+}e_{0}^{n} +h \sum_{i=1}^{N-1}\delta_{x}^{c}e_{i}^{n+\theta}\delta_{t}^{+}e_{i}^{n}+ \frac{h}{2}\delta_{x}^{-}e_{N}^{n+\theta}\delta_{t}^{+}e_{N}^{n}\right)\leq C(w_{d},\epsilon)\norm{\delta_{x}^{-}e^{n+\theta}}_{h}^{2}+ \epsilon\norm{\delta_{t}^{+}e^{n}}^{2},
		\end{align*}
	where \(\epsilon>0\) will be chosen later.
	
		Again, using the Young's inequality with Lemmas \ref{L4.3} and \ref{L3.6}, the third term on the right hand side of \eqref{4.17} is estimated by 
		\begin{align*}
			-\left(\phi(e^{n+\theta},e^{n+\theta}),\delta_{t}^{+}e^{n}\right)-\left(\phi(e^{n+\theta},W^{n+\theta}),\delta_{t}^{+}e^{n}\right)\leq C\norm{e^{n+\theta}}_{\infty}^{2}\norm{\delta_{x}^{-}e^{n+\theta}}_{h}^{2}+ C\norm{e^{n+\theta}}^{2}+ \epsilon\norm{\delta_{t}^{+}e^{n}}^{2}.
		\end{align*}
		Finally, from the last two term on the right hand side of \eqref{4.17}, we can see 
		\begin{align*}
			-\left(\phi(W^{n+\theta},e^{n+\theta}),\delta_{t}^{+}e^{n}\right)+(T^{n+\theta}, \delta_{t}^{+}e^{n})\leq C\norm{W^{n+\theta}}_{\infty}^{2} \norm{\delta_{x}^{-}e^{n+\theta}}_{h}^{2}+ \epsilon\norm{\delta_{t}^{+}e^{n}}^{2}+ C\norm{T^{n+\theta}}^{2}.
		\end{align*}
		Following the proof of Lemma \ref{L3.6}, we get from \eqref{4.17} with \(\epsilon=\frac{1}{6}\)
		\begin{align}
			\label{4.18}
			\nu \Big(\delta_{t}^{+}(\delta_{x}^{-}e^{n}),\delta_{x}^{-}e^{n+\theta} \Big)_{h}&+(c_{0}+w_{d})\delta_{t}^{+}e_{0}^{n} e_{0}^{n+\theta}+ (c_{1}+w_{d})\delta_{t}^{+}e_{N}^{n} e_{N}^{n+\theta}
			\nonumber\\&+ \frac{k^{3}}{18c_{0}}\big(\theta-\frac{1}{2}\big)^{3}(\delta_{t}^{+}e_{0}^{n})^{4}+ \frac{1}{36c_{0}}\delta_{t}^{+}(e_{0}^{n})^{4}
			\nonumber\\& +\frac{k^{3}}{18c_{1}}\big(\theta-\frac{1}{2}\big)^{3}(\delta_{t}^{+}e_{N}^{n})^{4}+ \frac{1}{36c_{1}}\delta_{t}^{+}(e_{N}^{n})^{4}
			+\frac{1}{2}\norm{\delta_{t}^{+}e^{n}}^{2}
			\nonumber\\&\leq C\Big(1+\norm{e^{n+\theta}}_{tr}^{2}+\norm{W^{n+\theta}}_{tr}^{2} \Big)\norm{\delta_{x}^{-}e^{n+\theta}}_{h}^{2}
			\nonumber\\&\quad +\frac{C}{k}\norm{e^{n+\theta}}_{tr}^{2}\Big((e_{0}^{n+1})^{2}+ (e_{N}^{n+1})^{2}\Big)
			+\frac{C}{k}\Big((e_{N}^{n+1})^{4}+ (e_{N}^{n})^{4}\Big)
			\nonumber\\&\quad+\frac{C}{k}\Big((e_{0}^{n+1})^{4}+ (e_{0}^{n})^{4}\Big)+\frac{C}{k}(e_{0}^{n})^{4}+C(e_{0}^{n})^{2}
			\nonumber\\&\quad +C\Big(\frac{(2\theta-1)^{2}}{k}+\frac{1}{k}\Big)(W_{0}^{n+\theta})^{2}(e_{0}^{n+1})^{2}+\frac{C}{k}(W_{0}^{n+\theta})^{2}(e_{0}^{n+1})^{4}
			\nonumber\\&\quad+ C\Big(\frac{(2\theta-1)^{2}}{k}+\frac{1}{k}\Big)(W_{N}^{n+\theta})^{2}(e_{N}^{n+1})^{2}+\frac{C}{k}(W_{N}^{n+\theta})^{2}(e_{N}^{n+1})^{4}
			\nonumber\\&\quad+\frac{C}{k}(e_{N}^{n})^{4}+\frac{{C}}{k}(e_{N}^{n})^{2}+C\norm{T^{n+\theta}}^{2}.
		\end{align}
		Note that
		\[\Big(\delta_{t}^{+}(\delta_{x}^{-}e^{n}), \delta_{x}^{-}e^{n+\theta}\Big)_{h}=\frac{1}{2}\delta_{t}^{+}\norm{\delta_{x}^{-}e^{n}}_{h}^{2}+ k(\theta-\frac{1}{2})\norm{\delta_{t}^{+}(\delta_{x}^{-}e^{n})}_{h}^{2},
		\]
		and
		\[\delta_{t}^{+}e_{0}^{n} e_{0}^{n+\theta}=\frac{1}{2}\delta_{t}^{+}(e_{0}^{n})^{2}+ k(\theta-\frac{1}{2})(\delta_{t}^{+}e_{0}^{n}).
		\]
		Multiplying \eqref{4.18} by \(2k\) and considering the case \(\theta\geq \frac{1}{2}\), we arrive at 
		\begin{align*}
			\nu \Big(\norm{\delta_{x}^{-}e^{n+1}}_{h}^{2}-\norm{\delta_{x}^{-}e^{n}}_{h}^{2}\Big)&+ (c_{0}+w_{d})\Big((e_{0}^{n+1})^{2}-(e_{0}^{n})^{2} \Big)+(c_{1}+w_{d})\Big((e_{N}^{n+1})^{2}-(e_{N}^{n})^{2} \Big)
			\\&+\frac{1}{18c_{0}}\Big((e_{0}^{n+1})^{4}-(e_{0}^{n})^{4} \Big)+\frac{1}{18c_{1}}\Big((e_{N}^{n+1})^{4}-(e_{N}^{n})^{4} \Big)+k\norm{\delta_{t}^{+}e^{n}}^{2}
			\\&\leq Ck\Big(1+\norm{e^{n+\theta}}_{tr}^{2}+\norm{W^{n+\theta}}_{tr}^{2} \Big)\norm{\delta_{x}^{-}e^{n+\theta}}_{h}^{2}
			\nonumber\\&\quad +C\norm{e^{n+\theta}}_{tr}^{2}\Big((e_{0}^{n+1})^{2}+ (e_{N}^{n+1})^{2}\Big)
			+C\Big((e_{N}^{n+1})^{4}+ (e_{N}^{n})^{4}\Big)
			\nonumber\\&\quad+C\Big((e_{0}^{n+1})^{4}+ (e_{0}^{n})^{4}\Big)+C(e_{0}^{n})^{4}+Ck(e_{0}^{n})^{2}
			\nonumber\\&\quad +C(2\theta-1)^{2}(W_{0}^{n+\theta})^{2}(e_{0}^{n+1})^{2}+C(W_{0}^{n+\theta})^{2}(e_{0}^{n+1})^{4}
			\nonumber\\&\quad+ C(2\theta-1)^{2}(W_{N}^{n+\theta})^{2}(e_{N}^{n+1})^{2}+C(W_{N}^{n+\theta})^{2}(e_{N}^{n+1})^{4}
			\nonumber\\&\quad+C(e_{N}^{n})^{4}+C\big(k+1\big)(e_{N}^{n})^{2}+Ck\norm{T^{n+\theta}}^{2}.
		\end{align*}
		Again, we follow the Lemmas \ref{L4.3} and \ref{L3.6}. Summing from \(n=0\) to \(M-1\) and using discrete Gronwall's inequality in the resulting inequality, the first part of the proof is completed with Theorem \ref{5.1}.
		
	Using the first part along with the inequality
	\[\norm{e^{n}}_{\infty}^{2}\leq 2\Big(\norm{\delta_{x}^{-}e^{n}}_{h}^{2}+ (e_{0}^{n})^{2}+(e_{N}^{n})^{2} \Big), \quad \forall \ 1\leq n\leq M,
		\]
		the second part of the proof follows.
	\end{proof}

	From Theorem \ref{5.1}, using truncation error  \eqref{eq5.2}, we obtain the following  error estimate for state variable in terms spatial and time step size. 
	\begin{corollary}
		\label{c5.1}
		Suppose all the hypotheses in Theorem \ref{5.1} are satisfied. Then, there exists a positive constant $ C $ independent of $ h $ and $ k $ such that for all $ 0 \leq n \leq M-1$
		\begin{align*}
			\norm{e^{M}}^{2}&+2k\nu \sum_{n=1}^{M-1} \norm{\delta_{x}^{-}e^{n+\theta}}_{h}^{2}+\frac{k}{\delta}\sum_{n=1}^{M-1}\left(({c_{0}}+{w_{d}})(e_{0}^{n+\theta})^{2}+\frac{2}{9c_{0}}(e_{0}^{n+\theta})^{4}\right)\\&+\frac{k}{\delta}\sum_{n=1}^{M-1}\left(({c_{1}}+{3w_{d}})(e_{N}^{n+\theta})^{2}+\frac{2}{9c_{1}}(e_{N}^{n+\theta})^{4}\right)\leq  Ce^{-2\alpha t_{M}}(h+(\theta-\frac{1}{2})k+ k^{2})^{2},
		\end{align*}
		where $ e^{n+\frac{1}{2}} = \frac{1}{2}(e^{n+1}+e^{n}).$
	\end{corollary}
	
	In the following theorem, we analyze the error of the feedback controllers.
	\begin{theorem}
		\label{5.2}
		Suppose all hypotheses in Theorem \ref{5.1} are satisfied. There exists a constant  $ C>0 $ independent of $ h $ and $ k $ such that for all $ 0\leq n \leq M $
		\begin{align*}
			|g_{i}(w_{i}^{n})-g_{i}({W_{i}^{n}})|\leq C\max_{0\leq n\leq M}\vert e_{i}^{n}\vert\leq Ce^{-2\alpha t_{M}}(h+k^{2}+(\theta-\frac{1}{2})k).
		\end{align*}
		where $ i=0$ and $ N $.
	\end{theorem}
	\begin{proof} We prove the results for $ i=0 $ .
		\begin{align}
			\label{4.16}
			g_{0}(w_{0}^{n})-g_{0}({W_{0}^{n}})=\frac{1}{\nu}\left((c_{0}+w_{d})e_{0}^{n}+\frac{2}{9c_{0}}e_{0}^{n}((w_{0}^{n})^{2}+ (W_{0}^{n})^{2}+w_{0}^{n}W_{0}^{n})\right).
		\end{align}
		Set $ w_{0}^{n}= e_{0}^{n}+W_{0}^{n} $ in (\ref{4.16}), we obtain
		\begin{align*}
			|g_{0}(w_{0}^{n})-g_{0}({W_{0}^{n}})|\leq\frac{1}{\nu}\left( (c_{0}+w_{d})|e_{0}^{n}|+\frac{2}{9c_{0}}|e_{0}^{n}||2(e_{0}^{n})^{2}+4 (W_{0}^{n})^{2}+e_{0}^{n}W_{0}^{n}|\right),
		\end{align*}
		which implies,
		\begin{align*}
			|g_{0}(w_{0}^{n})-g_{0}({W_{0}^{n}})|\leq C\max_{0\leq n\leq M}\vert e_{0}^{n}\vert,
		\end{align*}
		where the result follows from the corollary \ref{c5.1}.  The proof for $ i=N$ follows in a similar fashion.
	\end{proof}
\section{Numerical Results.}
In this section, we conduct numerical experiments to validate our theoretical findings. We examine both controlled and uncontrolled approximation solutions for various values of  $\theta \in \left[0, 1\right]$. We have shown that the approximate solution of \eqref{w11}-\eqref{w12} exponentially decay to zero with respect to the time, which shows that the numerical solution of \eqref{Eq1}-\eqref{Eq2} goes to a desired constant steady state solution. Furthermore, we show that the order of convergence for the state variable and feedback controllers. To solve the resulting nonlinear system of equations for the unknown solution $W^{n+1}$ at each time step, we utilize Newton's method, using the solution from the previous time step, $W^{n},$ as the initial guess.
\begin{example}
	In this example, we examine the initial condition $ (t=0) $ $w_{0}(x)=5x(x-1)-w_{d}$, where $w_{d} = 5$ is the steady state solution. We select $t=[0,1] $ and $\nu =1$. The difference scheme (\ref{th}) is solved for both the state cost and uncontrolled solutions.
	
	In Fig. \ref{fig:ex2}(a), with $ \theta=1,$  the approximate solution $W^{n}$ that does not converge to zero is referred to as the  ``uncontrolled solution". Upon applying feedback control, the approximate solution $W^{n}$ exponentially decay to zero for various values of $c_{0}$ and $c_{1}$ in the $ L^{2}$-norm.
	\begin{figure}[h!]
		\centering
		(a)	\includegraphics[width=0.43\textwidth]{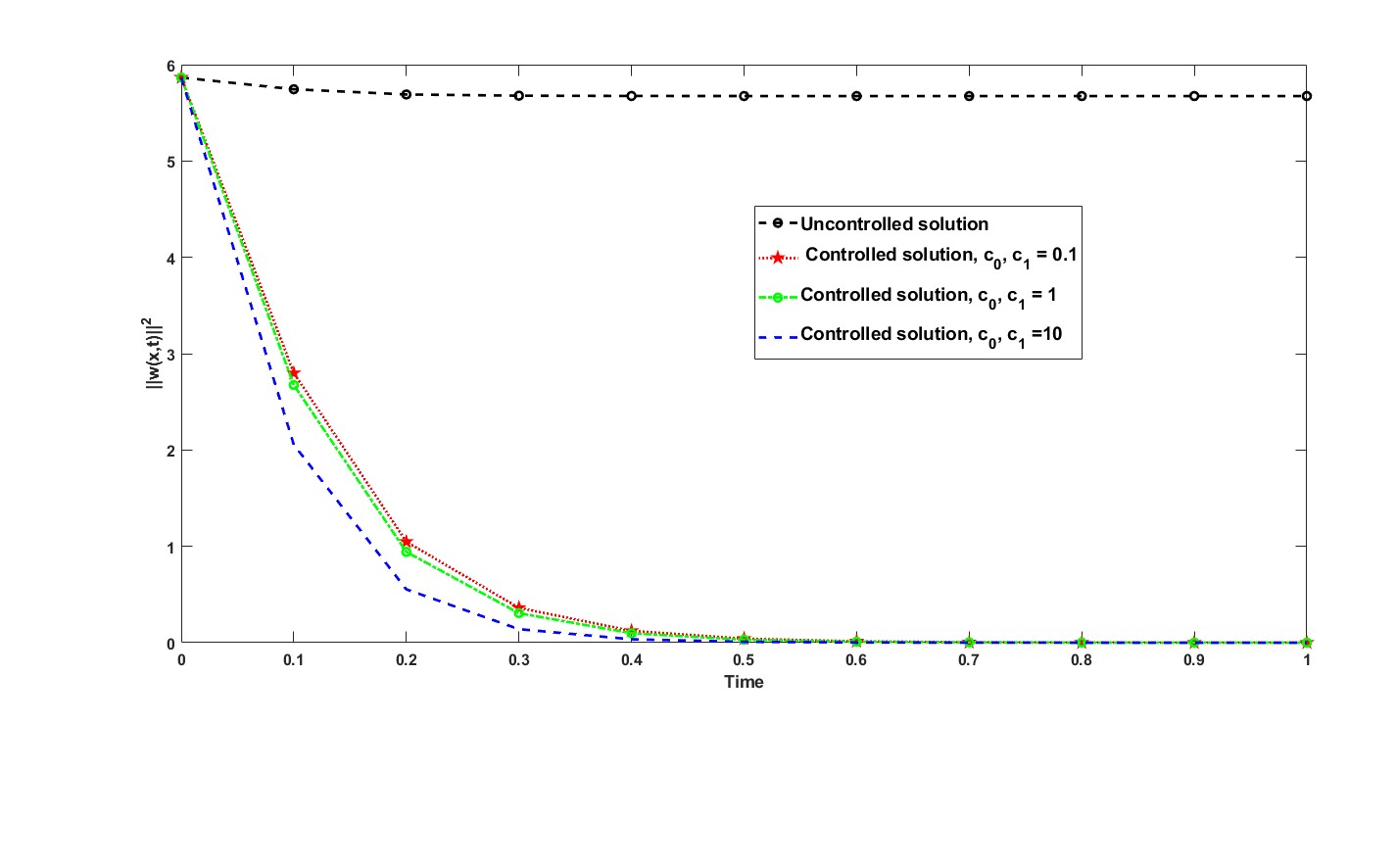}
		(b)\includegraphics[width=0.48\textwidth]{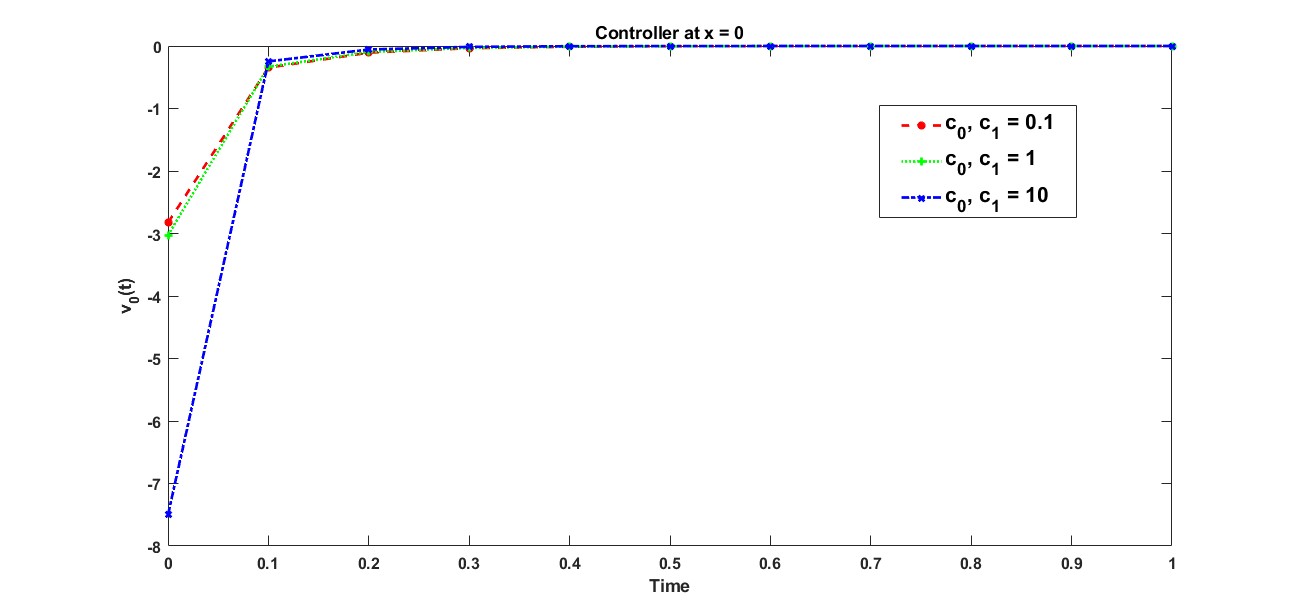}
		(c)	\includegraphics[width=0.35\textwidth]{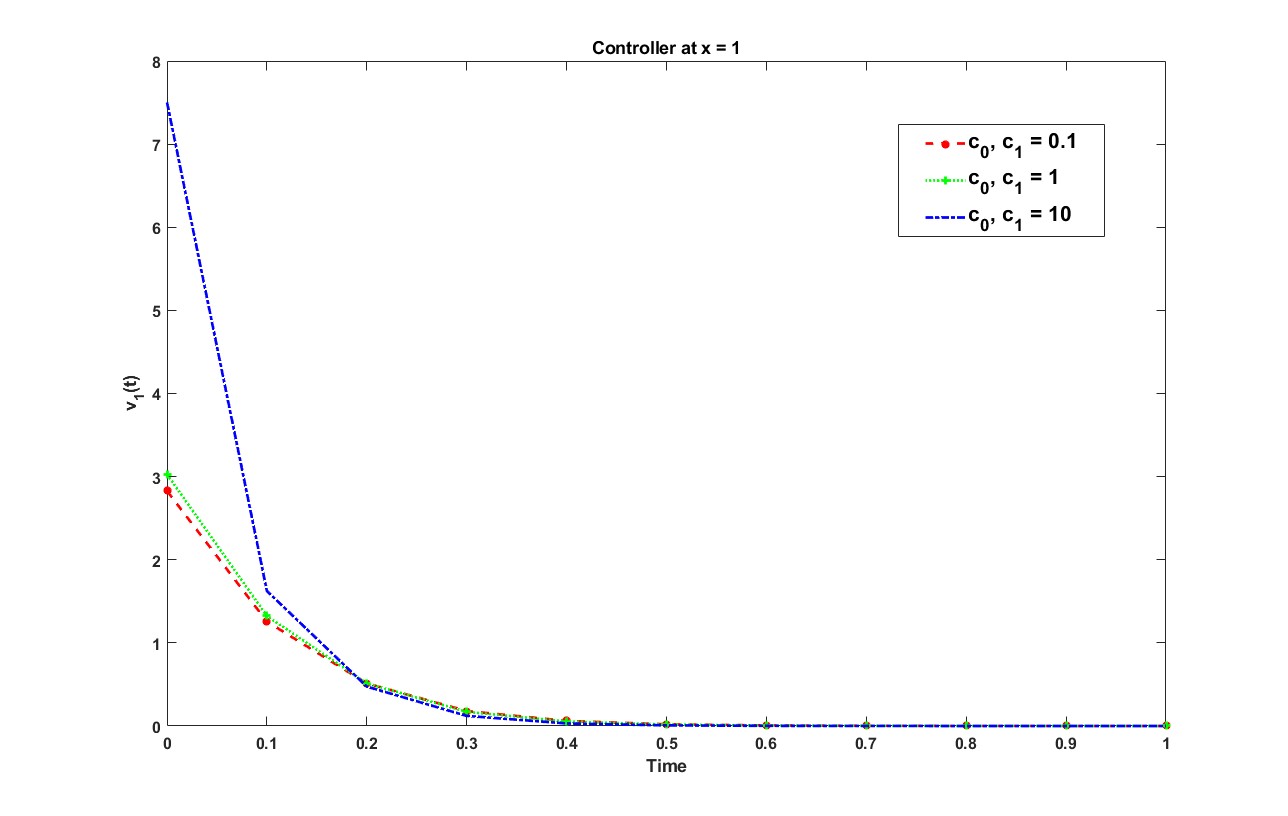}
		\caption{Example 5.1: {\bf (a)} Both Controlled and Uncontrolled solution in discrete $ L^{2}-$norm. {\bf (b)} Feedback controller at $x=0$ with $ \theta=1 $, {\bf (c)}  Feedback controller at $ x=1$  with $ \theta=1.$}\label{fig:ex2}
	\end{figure}
	
	 From Figs. \ref{fig:ex2}(b) and \ref{fig:ex2}(c), we observe that the feedback controllers at $ x=0 $ and $ x=1 $ for different values of $ c_{0} $ and $ c_{1} $ with $ \nu =1$ decay exponentially towards zero over time. The feedback controllers decay faster for large values of $ c_{i}, i=0, 1. $		
	
	For sufficiently small $ h,$  denote the error of the difference scheme (\ref{th}) in space and the order of convergence as follows
	\begin{align*}
		e_{\infty}(h) &= \max_{0\leq n\leq M}\max_{0\leq i\leq N}|W_{i}^{n}(h, k) - W_{2i}^{n}(\frac{h}{2}, k)|,\\&
		\text{Order of convergence (O. C.)}= \log_{2}\left(\frac{e_{\infty}(2h)}{e_{\infty}(h)}\right).
	\end{align*}
	Similarly, for temporal direction, we have found the order of convergence.
	
	\begin{table}[ht]
		\centering
		\caption{ The order of convergence for state variable with respect to space in Example 5.1 for $ \theta =1 $ with a fixed value of $ M= 10000. $ }
		\begin{tabular}{ c c c c c }
		\hline\noalign{\smallskip}
			\textbf{N} & $\norm{w^{n}-W^{n}}_{\infty}$       & \textbf{O. C.} & $\norm{w^{n}-W^{n}}$  &          \textbf{O. C.}   \\
			\noalign{\smallskip}\hline\noalign{\smallskip}
			$20$ &   $ --$            &     --       &  $-- $ &    --                 \\
			           
			$40$&   $   6.76e-07 $     &  $ -- $     &   $3.86e-07$ &    $--$           \\
		             
			$80$ &       $1.66e-07 $      & $2.03$       &    $9.46e-08$&     $2.03$              \\
			
			$160$ &      $  4.11e-08  $     &  $2.01$       &  $2.35e-08$  &  $2.01$           \\
		
			$320$ &        $1.03e-08 $    & $ 2.00 $         &  $5.87e-09$  &      $2.00$            \\
			
			$640$ &        $ 2.56e-09$    & $2.00$         &  $1.47e-09$  &      $2.00$            \\
			\noalign{\smallskip}\hline
		\end{tabular}
		\label{table:1}
	\end{table}
	
	\begin{table}[ht]
		\centering
		\caption{Example 5.1: The order of convergence for state variable in the temporal direction  when $ \theta=\frac{1}{2} $, and  $ \theta =1 ,$ with a fixed value of $ N=100. $  }
		\begin{tabular}{ c c c c c }
		\hline\noalign{\smallskip}
			\textbf{M} &  $\norm{w^{n}-W^{n}}_{\infty}$  & \textbf{O. C. ( $\theta =\frac{1}{2} $)}  & $\norm{w^{n}-W^{n}}$ &    \textbf{ O. C. ($ \theta =1) $ }                     \\
		\noalign{\smallskip}\hline\noalign{\smallskip}
			$100$ &     --                   &    --   &     --   &     --                                     \\
		        
			$200$&    $ 4.64e-06 $         &      --   &  $0.0001151$       &   --                               \\
		             
			$400$ &    $5.89e-07$      &    $ 2.97 $ &  $4.96e-05$ &       $ 1.22 $                             \\
		
			$800$ &     $1.48e-07$    &    $ 1.99$  &    $2.29e-05$ &     $ 1.11 $                             \\
			
			$1600$ &    $3.73e-08$    &  $ 1.99$   &   $1.10e-05$ &      $ 1.06 $                                \\
		
			$3200$ &   $9.35e-09 $    &  $1.99$   &    $5.40e-06$&       $ 1.03 $                               \\
			\noalign{\smallskip}\hline
		\end{tabular}
		\label{table:2}
	\end{table}

	\begin{table}[ht]
		\centering
		\caption{  The order of convergence of the feedback controllers with respect to space in Example 5.1 for $ \theta =1 $ with a fixed value of $ M= 10000. $ }
		\begin{tabular}{ c c c c  c }
			\hline\noalign{\smallskip}
			\textbf{N} & \textbf{Error in $L^{\infty}-$norm}  & \textbf{O. C. at $ x=0 $} &  \textbf{Error in $L^{\infty}-$norm}  & \textbf{O. C. at $ x=1 $ } \\                            
			\noalign{\smallskip}\hline\noalign{\smallskip}
			$40$ &   $ --$             &      $--$&   $ --$      &       $--  $                                      \\
		           
			$80$&   $ 1.816 $           &    $ -- $  &    $1.808$  &         $--  $                                         \\
			               
			$160$ &     $0.804$       &   $ 1.17$  &   $ 0.799$ &      $ 1.18$                                             \\
			
			$320$ &      $0.246 $       &    $1.71 $  &  $0.244$ &    $ 1.71$                                                \\
		
			$640$ &      $0.065 $       &    $1.92 $  &  $0.065$ &    $ 1.92$                                                \\
			
			$1280$ &        $0.017$      &  $ 1.98$   &   $ 0.016$  &    $1.98$                                               \\
		
			$2560$ &        $0.004$     &  $ 1.99 $  &  $ 0.004$    &     $ 1.99 $                                              \\
			
			$6120$ &        $0.0010$     &  $ 1.99 $  &  $ 0.0010 $    &     $ 1.99 $                                              \\
			\noalign{\smallskip}\hline
		\end{tabular}
		\label{table:x=0}
	\end{table}

	Table \ref{table:1}  presents the order of convergence of the state variable for $\theta=1$ in the discrete $L^2$ and $L^\infty-$norms. Here, we observe that the error for the state variable decreases in various norms and achieve second-order convergence. However, our difference scheme \eqref{th} is overall first-order accurate in space because we discretize the Neumann boundary with first-order accuracy.
	For other values of $ \theta, $ we obtain similar order of convergence for the state variable in space.

	In Table \ref{table:2}, the order of convergence of the state variable in the temporal direction is presented for $\theta=\frac{1}{2}$ and $\theta=1$. The results indicate that the  errors decrease in the $L^{\infty}-$norm. Specifically, we observe second-order convergence  when $\theta=\frac{1}{2}$ and first-order convergence when  $\theta=1$, as predicted in  Corollary \ref{c5.1}.
	For values of $\theta\in (\frac{1}{2},1), $ we observe first-order convergence in the temporal direction.	
	
	Table \ref{table:x=0} displays the order of convergence of feedback controllers in the $ L^{\infty} $-norm, which is found to be $ 2 $ in space. Moreover, in temporal direction, we find that the order of convergence the feedback controller is $ 1, $ when $ \theta =1; $ and $ 2,$ when $ \theta= \frac{1}{2}. $
\end{example}

In the second example, we  illustrate  the solution of the difference scheme (\ref{th}) with $\theta=\frac{1}{2}$, representing the Crank-Nicolson type difference scheme.  Furthermore, we examine the behavior of the feedback controllers for different values of $\nu$ and $ \theta $.
\begin{example}
	Here, we consider the initial condition  $ w_{0}(x)=2\cos(\pi x) -w_{d} $ at $ t=0 $ where $ w_{d}=3 $ is a steady state solution of   (\ref{Eq1}) - (\ref{Eq2}) and  $ x\in(0,1), \  t=[0,1] .$  We solve (\ref{w11}) - (\ref{w12}) with zero Neumann boundary conditions, known as the uncontrolled solution. The viscosity coefficient is set to $ \nu=0.1$. 
	
	From Fig. \ref{fig:ex1}(a), with $ \theta=\frac{1}{2}$, we observe that the trajectory of the approximate solution $ W^{n}$ in the discrete $ L^{2}-$norm does not converge to zero, which is referred to as the ``uncontrolled solution".  In contrast, when feedback control law is applied, the approximate solution $ W^{n} $ in the discrete $ L^{2}-$norm converges exponentially to zero for different values of $ c_{0} $ and $ c_{1}$.
	
	\begin{figure}[!h]
		\centering
		(a)	\includegraphics[width=0.44\textwidth]{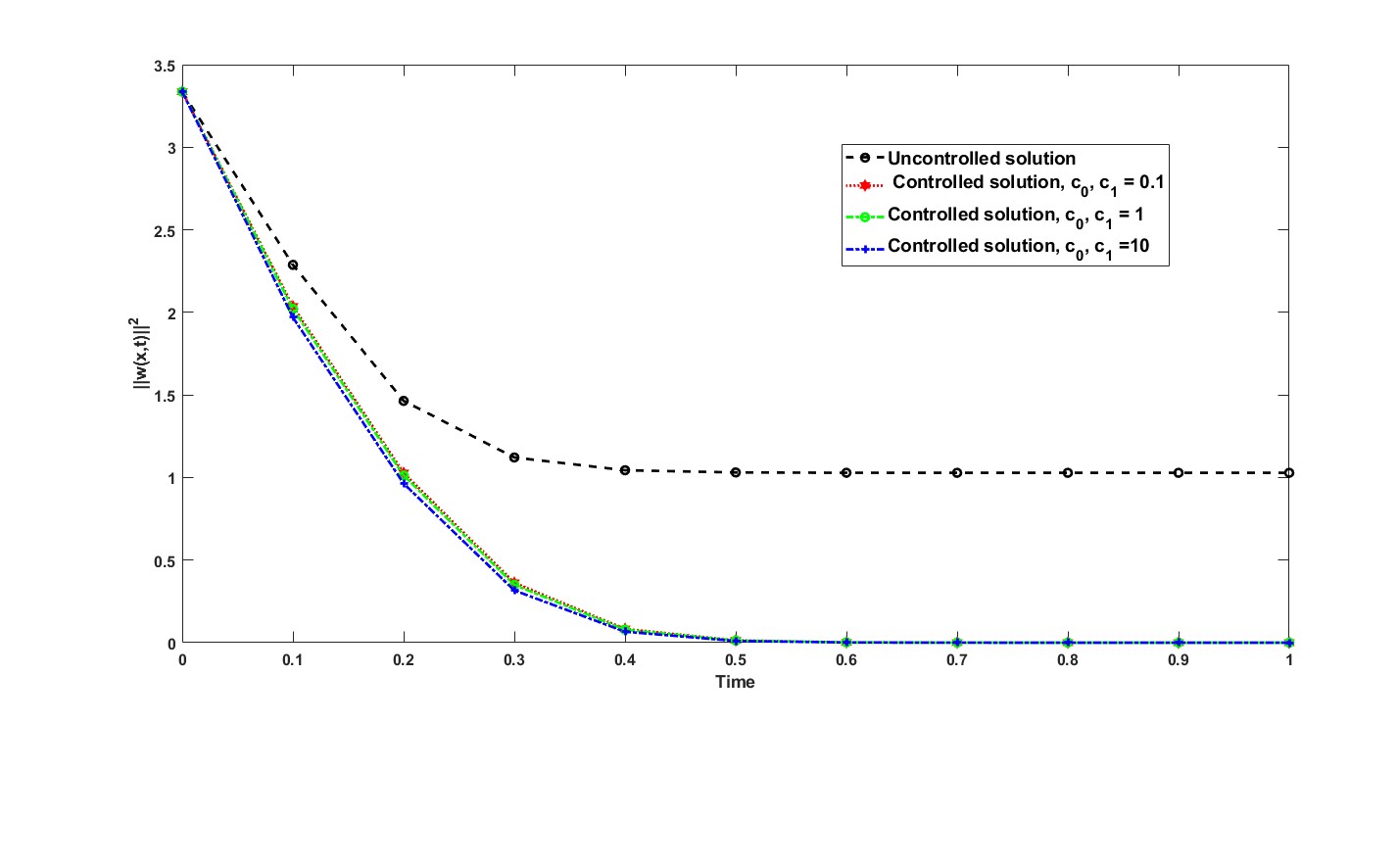}
		(b)\includegraphics[width=0.48\textwidth]{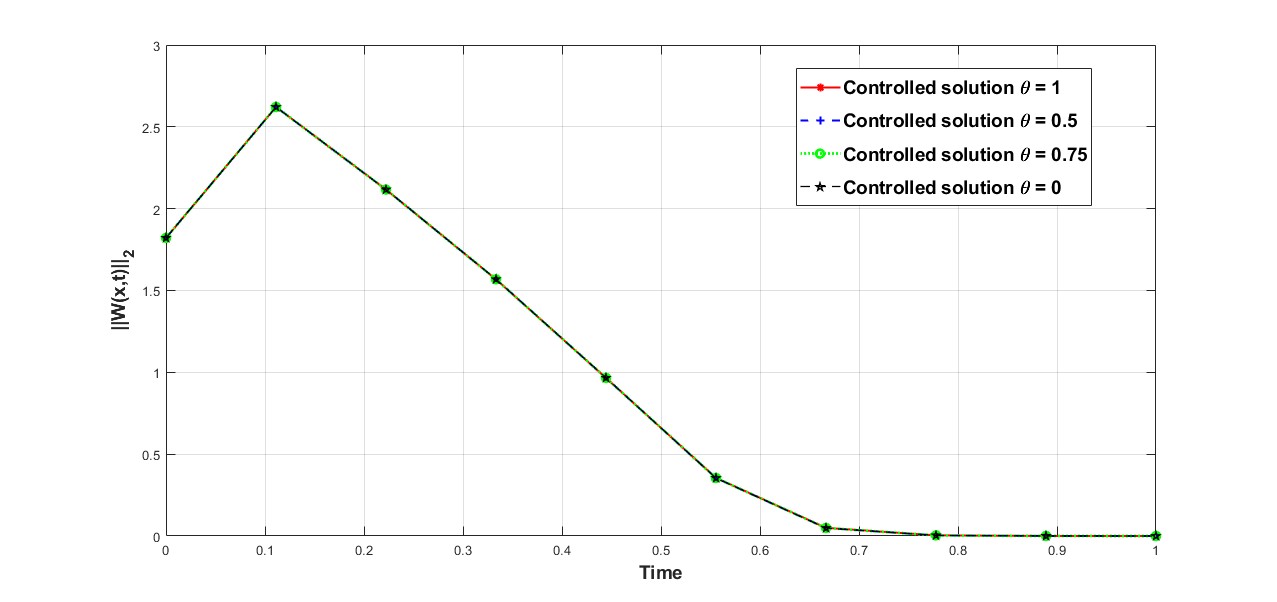}
		(c)	\includegraphics[width=0.45\textwidth]{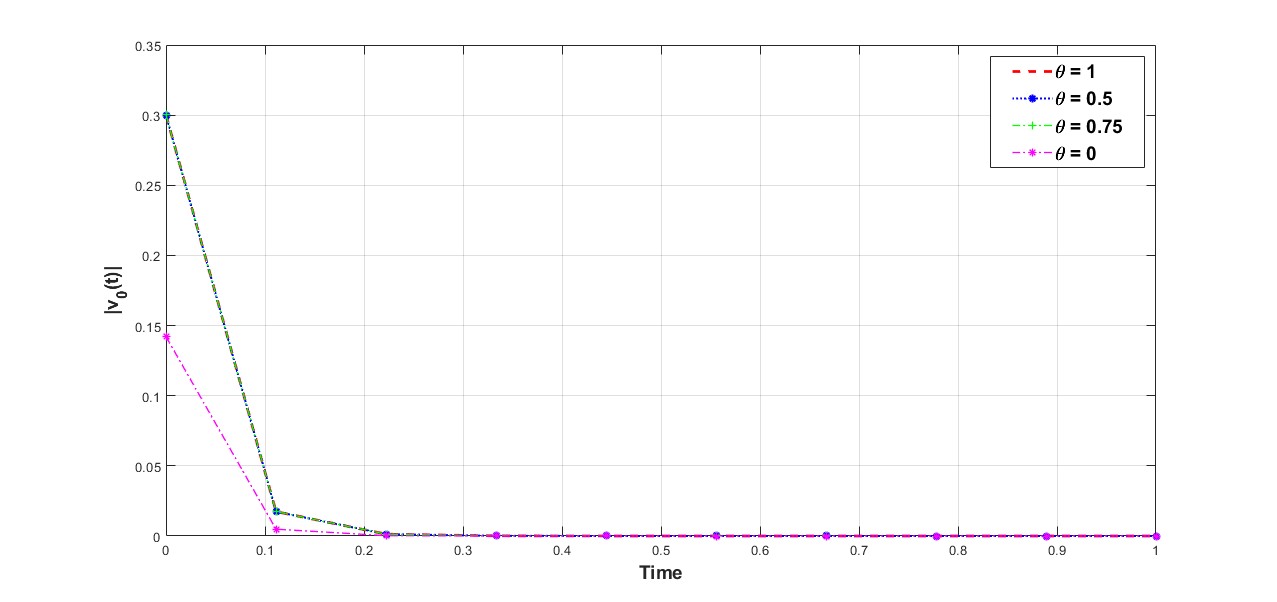}
		(d)\includegraphics[width=0.46\textwidth]{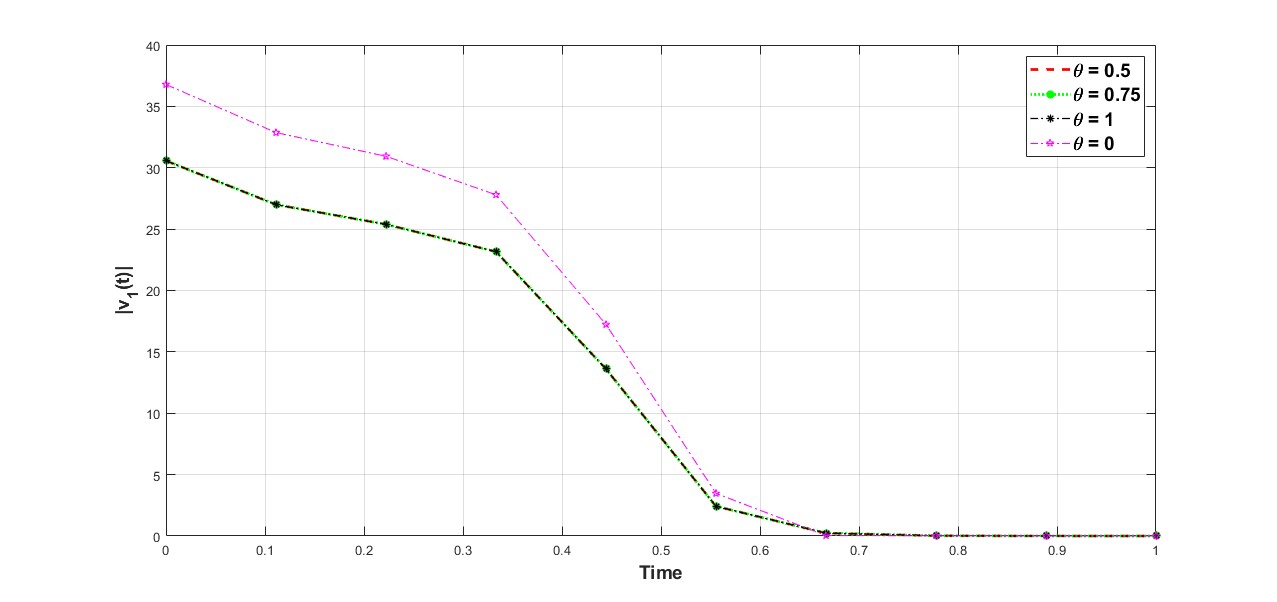}
		(e)	\includegraphics[width=0.45\textwidth]{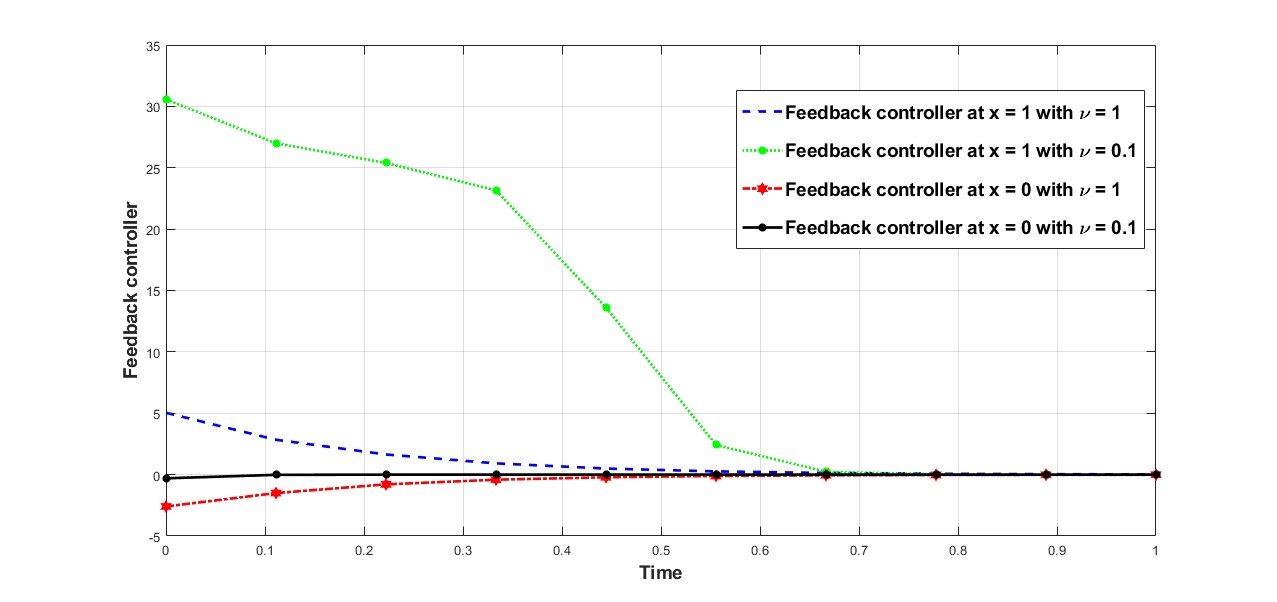}
		\caption{Example 5.2: {\bf (a)} Both Uncontrolled  and  Controlled solution in the discrete $ L^{2}-$norm  with different values of $ c_{0}$ and $ c_{1},$ {\bf (b)} Controlled solution in the discrete $ L^{2}-$norm with various values of $ \theta$, {\bf (c)}  Absolute value of the feedback controllers  with different  values of $\theta$ at $ x=0 $, {\bf (d)} Absolute value of the feedback controller  with different values of  $\theta$ at $ x=1$, {\bf (e)}  Feedback controller with different values of $ \nu $.}\label{fig:ex1}
	\end{figure}
	
	Figs. \ref{fig:ex1}(b)-(d) show the behavior of the state variable in the discrete $ L^{2}-$norm and the absolute value of feedback controllers for various values of $\theta$ with   $c_0=1$, and $c_1=1.$  For different values of $\theta$, the behavior of the state variable   is largely similar. However, when  $ 0\leq \theta<\frac{1}{2}$ and $ \theta\geq\frac{1}{2}$, the behavior of the absolute values of the feedback controllers at  $x=0$ and $x=1$  differ. Additionally, we observe that the feedback controllers approach zero over time.
	
	Fig. \ref{fig:ex1}(e) illustrates the trajectory of the feedback controller for different values of the parameter $\nu$  with fixed values of $c_{0}, c_{1} = 1,$ and $\theta = \frac{1}{2}$  over the time interval  $t=[0,1]$. The feedback controllers at $x=0$ and $x=1$ decrease to zero as time progresses. We observe that the feedback controller at $ x=1 $ lead faster decay for large value of $\nu$ with respect to the time.	
\end{example}

\textbf{Concluding Remarks:} In this work, we developed a $\theta$-based finite difference scheme for the viscous Burgers' equation with nonlinear Neumann boundary feedback control. Using the discrete energy method, we proved unconditional stability for $\theta \in [\frac{1}{2},1]$, ensuring bounded numerical solutions independent of the time step size, while conditional stability holds for $\theta \in [0,\frac{1}{2})$. The error analysis demonstrated first-order spatial accuracy of the state variable; however, the numerical experiments showed second-order convergence in space, an observation that warrants further theoretical explanation. The temporal convergence was found to be of second order for $\theta = \tfrac{1}{2}$ and of first order otherwise. A rigorous error estimate for the control variable was also established. All these results are shown to preserve the exponential stability property. Numerical experiments validated and illustrated the effectiveness of the proposed approach.

As future work, we intend to extend this framework to the viscous generalized Burgers' equation~\cite{MR4149614} with Neumann boundary feedback, and to explore high-order compact finite difference schemes for nonlinear parabolic PDEs under such boundary control settings.
\section*{Acknowledgments.}
Sudeep Kundu gratefully acknowledges the support of the Science \& Engineering Research Board (SERB), Government of India, under the Start-up Research Grant, Project No. SRG/2022/000360.

\section*{Declarations}
\textbf{Data Availability.} 

The codes  are available from authors on reasonable request. 

\textbf{CONFlLICT OF INTEREST.} 

The authors declare no conflict of interest.

	

\end{document}